\documentclass[final,leqno]{siamltex1213}

\usepackage{amsmath,amssymb,amscd,fancybox,ifthen,float,epsfig,subfigure}
\usepackage{graphicx}
\usepackage{xcolor}
\usepackage[all]{xy}
\usepackage{comment}
\usepackage{times}
\usepackage[normalem]{ulem}

\newtheorem{remark}[theorem]{Remark}

\newcommand\bysame{\leavevmode\vrule height 2pt depth -1.6pt width 23pt}

\newcommand\oS{\overline{S}}

\newcommand\LK{\operatorname{L^{Kim}}}
\newcommand\LKs[1]{\operatorname{L^{Kim}_{#1}}}
\newcommand\WF{\operatorname{WF}}

\newcommand\Kim{\operatorname{Kim}}

\newcommand\dist{\operatorname{dist}}

\newcommand\Int{\operatorname{int}}

\newcommand\hn{\widehat{n}}
\newcommand\fone{f^{(1)}}
\newcommand\ftwo{f^{(2)}}
\newcommand\fk[1]{f^{(#1)}}

\newcommand\cI{\mathcal I}
\newcommand\cJ{\mathcal J}

\newcommand\cP{\mathcal P}
\newcommand\dcP{\dot{\mathcal P}}

\newcommand\bzero{\boldsymbol 0}
\newcommand\btwo{\boldsymbol 2}

\newcommand\ty{\widetilde{y}}
\newcommand\tC{\widetilde{C}}

\newcommand\tg{\widetilde g}
\newcommand\tv{\widetilde v}

\newcommand\hv{\widehat v}

\newcommand\tL{\widetilde L}

\newcommand\bb{\boldsymbol b}

\newcommand\cE{\mathcal E}

\newcommand\tSigma{\widetilde \Sigma}

\newcommand\tb{\widetilde{b}}
\newcommand\tV{\widetilde{V}}

\newcommand\tu{\widetilde{u}}

\newcommand\tR{\widetilde{R}}

\newcommand{\la}{\langle}
\newcommand{\ra}{\rangle}

\newcommand\cC{\mathcal{C}}

\newcommand\hi{\hat i}

\newcommand\tf{\tilde{f}}

\newcommand\bbC{\mathbb C}

\newcommand\bbN{\mathbb N}

\newcommand\bbR{\mathbb R}

\newcommand\pa{\partial}

\newcommand\dcC{\dot{\mathcal C}}
\newcommand\restrictedto{\!\!\upharpoonright}

\renewcommand\supp{\operatorname{supp}}

\newcommand\ignore[1]{}

\title{Eigenfunctions and the Dirichlet problem for the\\
Classical Kimura Diffusion Operator}

\author{
  Charles L. Epstein\thanks{Department of Mathematics,
    University of Pennsylvania, 209 South 33rd Street, Philadelphia,
    PA 19104-6395; cle@math.upenn.edu. Research supported in part by
    the NSF under grants DMS12-05851, and  DMS-1507396; and by the ARO under grant
    W911NF-12-1-0552.}
  \and
  Jon Wilkening\thanks{Department of
    Mathematics, 970 Evans Hall, University of California, Berkeley,
    California 94720-3840; wilken@math.berkeley.edu. Research
    supported in part by 
    the US
    Department of Energy, Office of Science, Applied Scientific
    Computing Research, under award number DE-AC02-05CH11231.}
}

\date{January 8, 2016}

\begin{document}

\maketitle

\begin{abstract}
  We study the classical Kimura diffusion operator defined on the $n$-simplex,
  $$\LK=\sum_{1\leq i,j\leq
    n+1}x_i(\delta_{ij}-x_j)\pa_{x_i}\pa_{x_j},$$
  which has important applications in Population Genetics.  Because it
  is a degenerate elliptic operator acting on a singular space,
  special tools are required to analyze and construct solutions to
  elliptic and parabolic problems defined by this operator. The
  natural boundary value problems are the ``regular'' problem and the
  Dirichlet problem. For the regular problem, one can only specify the
  regularity of the solution at the boundary. For the Dirichlet
  problem, one can specify the boundary values, but the solution is
  then not smooth at the boundary.  In this paper we give a
  computationally effective recursive method to construct the
  eigenfunctions of the regular operator in any dimension, and a
  recursive method to use them to solve the inhomogeneous equation. As
  noted, the Dirichlet problem does not have a regular solution. We
  give an explicit construction for the leading singular part along
  the boundary. The necessary correction term can then be found using
  the eigenfunctions of the regular problem.
\end{abstract}
\pagestyle{myheadings}
\thispagestyle{plain}
\markboth{C.L.~Epstein and J.~Wilkening}{Eigenfunctions and the Dirichlet Problem
  for  Kimura Diffusion}

\section{Introduction}
An $n$-simplex, $\Sigma_n,$ is the subset of $\bbR^{n+1}$ given, in the \emph{affine
  model}, by the relations:
\begin{equation}
  x_1+\cdots+x_{n+1}=1\text{ with }0\leq x_i\text{ for }1\leq i\leq n+1.
\end{equation}
In population genetics problems, a point
$(x_1,\dots,x_{n+1})\in\Sigma_n$ is often thought of as representing the
frequencies of $n+1$ alleles or types. More generally, $\Sigma_n$ can
be thought of as the space of atomic probability measures with $n+1$
atoms. Mathematically, $\Sigma_n$ is a \emph{manifold with corners}; its
boundary is a stratified space made up of simplices of dimensions between $0$
and $n-1.$

The Kimura diffusion operator, which acts on functions of $n+1$ variables,
\begin{equation}\label{eq:LK:def}
  \LK=\sum_{1\leq i,j\leq
    n+1}x_i(\delta_{ij}-x_j)\pa_{x_i}\pa_{x_j}
\end{equation}
appears in the infinite population limit of the $(n+1)$-allele Wright-Fisher
model. It represents the limit of the random mating term, and actually appears
in the infinite population limits of many Markov chain models in population
genetics, see~\cite{Ewens,Kimura1964,griffiths1979,DerThesis}.
The Kimura diffusion operator has many
remarkable properties, which we employ in our analysis. The properties of this
operator reflect the geometry of the simplex in much the same way as the
standard Laplace operator reflects the Euclidean geometry of $\bbR^n.$

To include the effects of mutation, selection, migration, etc. the operator is
modified by the addition of a vector field
\begin{equation}
  V=\sum_{j=1}^{n+1}b_j(x)\pa_{x_j},
\end{equation}
which is tangent to $\Sigma_n,$ and inward pointing along the boundary of the
simplex. In applications to population genetics, the coefficient functions
$\{b_j(x)\}$ are typically polynomials. Linear terms usually suffice to model
migration and mutational effects, whereas higher order terms are needed to
model selection, see~\cite{Kimura1964,Li:1977,shimakura2}.

There are many statistical quantities of interest in population genetics that
can be computed by solving boundary value problems of the form
\begin{equation}\label{eqn1.3.01}
  (\LK+V) u= f\text{\; in\; }\Int\Sigma_n,\text{\; with\; }u\restrictedto_{b\Sigma_n}=g.
\end{equation}
For example, the probability of a path of the underlying process exiting through
a given portion of the boundary, or the expected first arrival time at a 
portion of the boundary, are expressible as the solutions of such boundary value
problems.  Examples of this type can be found in~\cite{littler1975-1}, and are
further discussed below.

A method for solving some of these problems, at least in principle, is given
in~\cite{shimakura1}, though it is not very explicit.  In this note we
introduce a computationally effective method for solving high-dimensional,
inhomogeneous regular and Dirichlet problems for the Kimura operator itself,
i.e. with $V=0.$ For the regular problem, one specifies minimal regularity
requirements for the solution at the boundary. For the Dirichlet problem, the
boundary values are specified, but the solution is then not smooth at the
boundary.  Our method also clarifies the precise regularity of the Dirichlet
solution in the closed simplex, at least when $f$ and $g$ are sufficiently
smooth. In addition we show, in somewhat greater generality, how to find the
eigenfunctions of these operators, which are represented as products of
functions of single variables, and how to compute the expansion coefficients.

The operator, $\LK+V,$ and variants thereof, appear in many classical papers in
population genetics,
see~\cite{Ewens,Fisher1922,Kimura1955_1,Kimura1955_2,Kimura1955_3,Wright1931}. Recently,
there has been a resurgence of interest in using the Kimura diffusion equation
as a forward model for maximum likelihood estimators of selection coefficients,
demographic models, mutation rates, effective population sizes, etc.  The
evolution of other observable measures of genetic variability such as the
allele frequency spectrum, or site frequency spectrum, can also be shown to
satisfy a variant of the Kimura diffusion equation,
see~\cite{EvansSlatkinShvets, BustamanteEtAl}.  To use this diffusion process
as a forward model, one either needs to have efficient means for solving the
Kimura diffusion equation,
see~\cite{BustamanteEtAl,SBS,Bollback:2008br,Lacerda:2014gi}, or one must
simulate the underlying stochastic process,
\cite{EvansSlatkinShvets,EvansSchraiberSlatkin}.  In most previous work where
the Kimura diffusion equation is solved, the underlying space is
1-dimensional. Even in one dimension, many authors employ numerical methods
that rely on finite difference approximations. These are, however, not reliable
for imposing the subtle boundary conditions that arise with degenerate
operators like $\LK,$ and can in fact lead to errors; see~\cite{KTS}, and the
supplement to~\cite{BWS}. By contrast, our approach provides a stable
construction, mathematically equivalent to a Gram-Schmidt procedure, of bases of
eigenfunctions. These can be used to accurately solve both elliptic and
parabolic problems, as well as compute approximations to the heat kernel itself,
which are of central importance in a variety of applications;
see~\cite{Song:2012cb,SBS,Steinrucken:9999ib}.

Our construction for the polynomial eigenfunctions of $\LK+V$ is
applicable provided that the operator has ``constant weights,''
see~\eqref{constwts} below.  The case of positive weights has been
studied extensively in the literature. Kimura \cite{kimura1956} and
Karlin and McGregor \cite{karlin:64} used hypergeometric functions to
study the two-dimensional case.  Littler and Fackerell
\cite{littler1975} generalized Karlin and McGregor's approach to
higher dimensions using bi-orthogonal polynomial systems. Griffiths
\cite{griffiths1979} showed that the polynomial eigenfunctions of the
diffusion operator corresponding to a repeated eigenvalue can be
orthogonalized and grouped together to form reproducing kernel
polynomials that appear in the transition function expansion of the
diffusion process. One way to represent the orthogonal polynomials
in the reproducing kernels is via a triangular
construction of Proriol \cite{proriol:57} and Koornwinder
\cite{koornwinder:75} for multivariate
Jacobi polynomials. Griffiths and Span\'o also discuss this
construction \cite{griffiths2010}, and provide probabilistic
connections to multivariate versions of several families of classical
orthogonal polynomials, including the Jacobi polynomials that arise
here \cite{griffiths2011}.  In the present work, we present a direct
construction of polynomial eigenfunctions for the $V=0$ case, so that
it is not necessary to take limits of the positive weight case as the
mutation rates approach zero \cite{griffiths1979}.  One novelty that
arises is that when $V=0$, as functions on the $n$-simplex, these
eigenfunctions do not belong to a single $L^2$-space, but each belongs
to an $L^2$-space of some stratum of the boundary. Their coefficients
in the representation of a function in terms of this spanning set are
computed as inner products on these lower dimensional strata.

The constructions presented here can serve as the foundation
for a perturbative method for solving Kimura-type diffusions
with a more complicated vector field, also modeling selection.  We
will return to these and other elaborations of the theory presented
here in a subsequent publication.

\vspace*{10pt} {\small \centerline{\bf Acknowledgements} CLE is grateful for
  many useful and informative discussions with Camelia Pop, Rafe Mazzeo, and
  Yun Song on questions connected to the contents of this paper. We are also
  very grateful for the careful reading and detailed suggestions made by the
  referees of this paper.}

\section{Some Facts about the Kimura Diffusion Operator}
We begin our analysis by reviewing some of the remarkable properties of $\LK.$
A very important fact about $\LK$ is the result of K.~Sato \cite{Sato}, which
states that if $U$ is a $\cC^2$--function that vanishes on $\Sigma_n,$ then
$[\LK U]\restrictedto_{\Sigma_n}=0$ as well.  Thus we can start with a function
$U(x_1,\dots,x_{n+1})$ defined on $\Sigma_n,$ and extend it to be independent
of any one of its arguments, say $x_j.$ This uniquely defines a function on the
projection of the simplex to the hyperplane $\{x_j=0\},$ and vice versa. For
definiteness we take $j=n+1.$ This gives a function
%
\begin{equation}
  u(x_1,\dots,x_n) = U\big(x_1,\cdots,x_n,1-(x_1+\cdots+x_n)\big)
\end{equation}    
defined on a \emph{projective model} of the simplex:
\begin{equation}
  \tSigma_n=\{x\in\bbR^n:\: x_1+\cdots+x_n\leq 1\text{ with }\,x_i\ge 0\,\text{
    for }\,i=1,\dots,n\}.
\end{equation}
To compute $\LK U\restrictedto_{\Sigma_n},$ we can apply $\LK$ in $n$-variables
to $u:$
\begin{equation}
  \LK U(x_1,\dots,x_{n+1})\restrictedto_{\Sigma_n}=\sum_{1\leq i,j\leq
    n}x_i(\delta_{ij}-x_j)\pa_{x_i}\pa_{x_j}u(x_1,\dots,x_n).
\end{equation}
Here and in the sequel, when using a projective model, we let
$x_{n+1}=1-(x_1+\dots+x_n).$ The projective model is useful for computations, whereas the affine model
shows that this operator is entirely symmetric under permutations of the
variables $(x_1,\dots,x_{n+1}).$  In particular, it makes clear that, from the
perspective of $\LK,$ all the vertices of $\Sigma_n$ are ``geometrically'' identical, something
that is not evident in the projective model. 

In this note we further investigate some remarkable properties of the operator
$\LK,$ which are hinted at in~\cite{shimakura1,shimakura2}. We first consider
a recursive construction of the polynomial eigenfunctions of $\LK.$ After
that we show how to solve the inhomogeneous Dirichlet problem for $\LK,$ given
in~\eqref{eqn1.3.01}.  It follows easily from Sato's theorem that this problem
does not generally have a solution in $\cC^2(\Sigma_n),$ even if $f$ and $g$
are both in $\cC^{\infty}.$ Our method of solution exhibits the precise form
of the singularities along the various boundary strata.

\subsection{Vector fields}
A generalized Kimura diffusion in $\Sigma_n,$ with ``standard'' second order
part, is a second order differential
operator of the form
\begin{equation}\label{eqn2.1}
  \tL=\sum_{1\leq i,j\leq n}x_i(\delta_{ij}-x_j)\pa_{x_i}\pa_{x_j}+\sum_{j=1}^n\tb_j(x)\pa_{x_j}.
\end{equation}
The vector field is normally required to be inward pointing, which means that
\begin{equation}
  \tb_j(x)\restrictedto_{x_j=0}\;\geq 0\text{ for }j=1,\dots,n,
\end{equation}
and
\begin{equation}
  \sum_{j=1}^n\tb_j(x)\restrictedto_{x_1+\cdots+x_n=1}\,\leq 0.
\end{equation}
We denote the first order terms by
\begin{equation}
  \tV=\sum_{j=1}^n\tb_j(x)\pa_{x_j}.
\end{equation}
In the affine model
\begin{equation}\label{eq:L:affine}
  L=\sum_{1\leq i,j\leq
    n+1}x_i(\delta_{ij}-x_j)\pa_{x_i}\pa_{x_j}+V, \qquad
  V=\sum_{j=1}^{n+1}b_j(x)\pa_{x_j},
\end{equation}
with the additional condition that
\begin{equation}\label{eq:vec:tan}
  \sum_{j=1}^{n+1}b_j(x)\restrictedto_{x_1+\cdots+x_{n+1}=1}\,=0.
\end{equation}

Considerable generalizations of this class of operators and spaces are
introduced in~\cite{EpMaz2}; though in the present work we concentrate 
on the classical cases of model operators on simplices and positive orthants,
$\bbR_+^n.$ In most of the Probability and Population Genetics literature a
different normalization is employed, namely $\LK = \frac 12\left(\sum
x_i(\delta_{ij}-x_j)\pa_{x_i}\pa_{x_j}\right)$.
In this paper we use the normalization more common in mathematical
analysis, which omits  the factor of $1/2$.

The boundary of $\Sigma_n$ is a stratified space.  If $K$ is a boundary face of
$\Sigma_n$ of codimension $n-k,$ then it is again a simplex and is represented,
in the affine model, by a subset of the form
$x_{i_1}+\cdots+x_{i_{k+1}}=1$  with $x_{i_l}\ge0$ for
$l=1,\dots,k+1,$ which implies that $x_{j_1}=\cdots=x_{j_{n-k}}=0.$ We set
$\cI=\{i_1,\dots,i_{k+1}\}$ and
\begin{equation}\label{eq:cJ}
    \cJ=\{j_1,\dots,j_{n-k}\}=\{1,\dots,n+1\}\setminus \cI.
\end{equation}
We  denote this boundary face by $K_{\cI}.$ 

Every boundary face is a simplex, and the formula for the operator
analogous to $\LK$ is the same for any boundary face. For example for $K_{\cI}$,
the sum in (\ref{eq:LK:def}) is simply  restricted to the variables
$\{x_{i_1},\dots,x_{i_{k+1}}\}.$ We denote this operator by
\begin{equation}
  \LKs{\cI} = \sum_{i,j\in\cI} x_i(\delta_{ij}-x_j)\pa_{x_i}\pa_{x_j}.
\end{equation}
Let
$v$ be a $\cC^2$-function on $K_{\cI}$ and let $\hv$ denote any
$\cC^2$-extension of $v$ to the ambient $\bbR^{n+1}.$ K.~Sato, in fact, proved, (see~\cite{Sato})
that 
\begin{equation}\label{eq:Ohta:cI}
  \LKs{\cI}v=(\LK \hv)\restrictedto_{K_{\cI}}.
\end{equation}
Thus, the restriction or extension of $\LK$ to boundary strata or
higher-dimensional simplices is canonical, similar to the connection
between the projective and affine models discussed at the beginning of
this section. The extension property makes it more natural to label
the ``last'' variable as $x_{n+1}$ rather than $x_0$, since it may not
actually be the last.

In the affine representation, a vector field $V$ (satisfying (\ref{eq:vec:tan}))
is tangent to $K_{\cI}$ provided that
\begin{equation}
  V x_j\restrictedto_{x_j=0}\,=0, \qquad\quad (j\in\cJ).
\end{equation}
Condition (\ref{eq:vec:tan}) requires that $V$ be tangent to
$\Sigma_n$ itself,
i.e.~$V(x_1+\cdots+x_{n+1})\restrictedto_{\Sigma_n}\,=0.$ For vector
fields, the analogues of Sato's results are obvious: if $U$ vanishes
on $\Sigma_n$ then $VU=0$, so extending a function $U$ defined on
$\Sigma_n$ to be constant in any coordinate direction leads to an
unambiguous value of $VU\restrictedto_{\Sigma_n}$; and, defining
\begin{equation}
  V_{\cI}=\sum_{i\in\cI} b_i(x)\partial_{x_i},
\end{equation}
we see that if $V$ is tangent to $K_{\cI}$ then $V_{\cI}v =
(V\hv)\restrictedto_{K_{\cI}}$, where $v$ and $\hv$ are defined as in
(\ref{eq:Ohta:cI}).


  \subsection{Constant weights and the Hilbert space setting}
  In addition to $\LK $ there are other classes of ``special'' Kimura diffusion
  operators. Classically one singles out operators with constant ``weights.''
  (This terminology is introduced in~\cite{EpMaz4}.) These operators have the
  property that the functions $b_j(x)$ (or $\tb_j(x))$ are linear and
\begin{equation}\label{constwts}
  (V x_j)\restrictedto_{x_j=0}\,=b_j(x)\restrictedto_{x_j=0}\,=b_j,
\end{equation}
a constant.  In the projective model, one may readily show that
  such a vector field takes the form
\begin{equation}\label{eq:Vtil:b}
  \tV_{\bb}=\textstyle\sum_{j=1}^n(b_j-Bx_j)\pa_{x_j},
\end{equation}
where
\begin{equation}
  \bb=(b_1,\dots,b_{n+1}) \text{\; and\; } B=b_1+\cdots+b_{n+1}.
\end{equation}
Note that $b_{n+1}$ enters (\ref{eq:Vtil:b}) through $B$.
We define
%
\begin{equation}
  \LKs{\bb} = \LK+V_{\bb}.
\end{equation}
These operators are special because they are self-adjoint with respect
to the $L^2$ inner product on $\tSigma_n$ defined by the following
measure, which (up to normalization) has the Dirichlet density
representing the stationary distribution of the underlying Markov
process (see e.g.~\cite{griffiths1979,griffiths2010,griffiths2011})
\begin{equation}
  d\mu_{\bb}(x)=w_{\bb}(x)\, dx_1\cdots dx_n, \qquad\quad
  w_{\bb}(x) = \prod_{j=1}^{n+1}x_j^{b_j-1}.
\end{equation}
Indeed, in this case $\LKs{\bb}$ in (\ref{eq:L:affine}) may be written
\begin{equation}\label{eq:L:sym}
  \LKs{\bb}u = \sum_{i<j}^{n+1}\frac{(\pa_{x_i} - \pa_{x_j})[w_{\bb} x_i x_j (\pa_{x_i} - \pa_{x_j})u]}{w_{\bb}},
\end{equation}
where the sum is over all pairs $i,j\in\{1,\dots,n+1\}$ with $i<j$, and,
  in the projective model
\begin{equation}\label{eq:L:sa}
  \la \widetilde{\LKs{\bb}} u,v\ra = \la u,\widetilde{\LKs{\bb}} v\ra = -\sum_{i<j}^{n+1} \int_{\tSigma_n} x_ix_j[(\pa_{x_i}
- \pa_{x_j})\hat{u}][(\pa_{x_i} - \pa_{x_j})\hat{v}]\,d\mu_{\bb}(x).
\end{equation}
Here $\hat{u}$, $\hat{v}$ are independent of $x_{n+1}$ and agree
  with $u$, $v$ on $\tSigma_n$. In deriving (\ref{eq:L:sa}), it was
  assumed that $w_{\bb}x_ix_j[(\pa_{x_i}-\pa_{x_j})u]v=0$ and
  $w_{\bb}x_ix_j[(\pa_{x_i}-\pa_{x_j})v]u=0$ on the faces $\{x_i=0\}$
  and $\{x_j=0\}$, so care must be taken in defining the domain of
  $\widetilde{\LKs{\bb}}$
  when working in the Hilbert space setting with $\bb=\mathbf{0},$ see~\cite{shimakura1}.
  Note that $x_{n+1}$ is treated as an independent variable in these
  formulas when computing partial derivatives, but
  $x_{n+1}=1-(x_1+\cdots+x_n)$ when evaluating integrals over
  $\tSigma_n$. A useful variant of (\ref{eq:L:sym}) is
\begin{equation}
  \widetilde{\LKs{\bb}} u = \sum_{i<j}^n
    \frac{(\pa_{x_i} - \pa_{x_j})[w_{\bb} x_i x_j (\pa_{x_i} - \pa_{x_j})u]}{w_{\bb}}
  + \sum_{i=1}^n \frac{\pa_{x_i}(w_{\bb} x_i x_{n+1} \pa_{x_i} u)}{w_{\bb}},
\end{equation}
where $x_{n+1}$ is now treated as a dependent variable, i.e.~$\pa
  x_{n+1}/\pa x_i=-1$.  If all the $b_j$ are positive, then
$d\mu_{\bb}(x)$ has finite total mass. In this paper we are
primarily interested in the case where all of the $b_j$
vanish, in which case the $\mu_{\bb}$-volume of $\Sigma_n$ is infinite.

\subsection{The Dirichlet problem and alternative function spaces}
In section~\ref{sec4} we present a method for solving the inhomogeneous
Dirichlet problem in~\eqref{eqn1.3.01}.  These results have extensions to the
case of Kimura diffusions with constant weights, though Dirichlet boundary
conditions are only appropriate on faces $\{x_j=0\}$ for which $b_j<1$.  For
simplicity, we focus our attention here on the case when all the weights are
zero, which is already of central importance in applications. We also assume
that $f$ and $g$ in~\eqref{eqn1.3.01} are sufficiently smooth. A different
analysis of this problem, employing blow-ups, appears
in~\cite{Jost1,Jost2,Jost3}. The blow-up approach gives a much less explicit
description of the singularities that arise when the data is smooth on the
simplex itself, but allows for considerably more singular data.

We use the notation and definitions of various function spaces given in
  the monograph~\cite{EpMaz2}. The principal symbol of the Kimura diffusion
  operator,
  \begin{equation}
   P^{\Kim}(\xi)= \sum_{i,j}x_i(\delta_{ij}-x_j)\xi_i\xi_j,
  \end{equation}
  defines the dual of the \emph{intrinsic} metric on the simplex. This
  corresponding metric is singular along the boundary and incomplete, i.e., the
  boundary is at a finite distance from interior points. The distance function
  on $\Sigma_n$ defined by this metric is equivalent to
  \begin{equation}\label{eqn1.23} 
    \rho_i(x,y)=\sum_{j=1}^{n+1}|\sqrt{x_j}-\sqrt{y_j}|.
  \end{equation}
  We also use two scales of anisotropic H\"older spaces,
  $\cC^{k,\gamma}_{\WF}(\Sigma_n)$ and
  $\cC^{k,2+\gamma}_{\WF}(\Sigma_n),$ $k\in\bbN_0, \gamma\in (0,1),$
  introduced in~\cite{EpMaz2}, with respect to which the operator
  $\LK$ has optimal mapping properties. These spaces are defined with
  respect to the intrinsic metric. For example
  $f\in\cC^{0,\gamma}_{\WF}(\Sigma_n)$ if $f\in\cC^0(\Sigma_n)$ and
  \begin{equation}
    [f]_{0,\gamma,\WF}=\sup_{x\neq y\in\Sigma_n}\frac{|f(x)-f(y)|}{\rho_i(x,y)^{\gamma}}<\infty.
  \end{equation}
  If $\lambda<0$ and $f\in\cC^{k,\gamma}_{\WF}(\Sigma_n),$ then the
  elliptic equation, $(\LK-\lambda)u=f,$ has a unique solution,
  $u\in\cC^{k,2+\gamma}_{\WF}(\Sigma_n),$ indicating that this is
  indeed the correct notion of ``elliptic regularity'' for operators
  of this general type.




\section{A 1-d Example}\label{sec2}
We begin by considering the Dirichlet problem in the 1d-case. These results are
well known, and serve as motivation for our subsequent development.
Suppose that we would like to find the solution to
\begin{equation}
  x(1-x)\pa_x^2u=f\text{\; with\; }u(0)=g_0,\, u(1)=g_1.
\end{equation}
We can write $u=u_0+(1-x)g_0+xg_1,$ where $u_0$ vanishes at the boundary of
$[0,1].$  It is apparent that $u_0$ cannot be $\cC^2$
up to the boundary unless $f(0)=f(1)=0.$ In fact,
\begin{equation}\label{eq:f0:1d}
  u_0(x) = x\log xf(0)+(1-x)\log(1-x) f(1)+\tu_0(x),
\end{equation}
where
\begin{equation}\label{eq:u0:1d}
  x(1-x)\pa_x^2\tu_0=f(x)-[(1-x)f(0)+x f(1)]
  \overset{\text{def}}{=}\fone\text{\; with\; }\tu_0(0)=\tu_0(1)=0;
\end{equation}
the right hand side, $\fone,$ is as smooth as $f,$ and vanishes at $0$ and $1.$
As was shown in~\cite{WF1d,EpMaz2}, this equation has a unique solution, with optimal
smoothness measured in the anisotropic H\"older spaces,
$\cC^{k,2+\gamma}([0,1]).$ In particular, if $f\in\cC^{\infty}([0,1])$ then so
is $\tu_0.$ 

The eigenfunctions of $ x(1-x)\pa_x^2$ that vanish at the boundary are
polynomials of the form
$x(1-x)p_m(x),$ where $p_m$ is the polynomial of degree $m\ge0$ that
satisfies the equation%
\begin{equation}\label{eq:eig:1d}
  \big[ x(1-x)\pa_x^2 + 2(1-2x)\pa_x -2 \big] p_m = \lambda_m p_m.
\end{equation}
For later reference, the left-hand side may also be written
$[\LKs{\btwo}-2]p_m$.
By inspecting the action on $\cP_d/\cP_{d-1}$, where $\cP_d$ is the
  space of polynomials of degree at most $d$, we see that
  $\lambda_d=-(d+1)(d+2)$.  Since the eigenfunctions are orthogonal with
  respect to $d\mu_{\bzero}(x)$, these polynomials are orthogonal with
  respect to the inner product
\begin{equation}\label{eq:inner:prod}
  \langle p,q\rangle=\int_{0}^1p(x)q(x)\,x^\alpha(1-x)^\beta dx,
\end{equation}
with $\alpha=\beta=1$. Thus, they are multiples of the corresponding Jacobi
polynomials \cite{szego,gautschi:book}, $p_d(x)\propto P_d^{(1,1)}(2x-1).$ If we
choose the normalization $\|p_d\|_{L^2([0,1];d\mu_{\btwo})}=1$, the result would be the same as performing
Gram-Schmidt on the monomials $\{1,x,x^2,\dots\}$. This
yields the 3-term recurrence
\begin{equation}\label{eq:recur}
\begin{aligned}
  & p_0(x)=\sqrt{1/\gamma_0},\qquad \sqrt{b_1}\, p_1(x)= (x-a_0)p_0(x) \\
  & \sqrt{b_{m+1}} \, p_{m+1}(x) = (x-a_m)p_m(x) - \sqrt{b_m}\, p_{m-1}(x), \qquad (m\ge1),
\end{aligned}
\end{equation}
where $\gamma_0=1/6$, $a_m=1/2$ and $b_m=\frac{m(m+2)}{4[4(m+1)^2-1]}$ in this case.
Solving (\ref{eq:u0:1d}) for $\tu_0$ now boils down to representing $f^{(1)}$ in
the eigenbases:
\begin{equation}\label{eq:expansion:1d}
  \tf(x) = \frac{f^{(1)}(x)}{x(1-x)} = \sum_{m=0}^\infty c_m p_m(x) \quad \Rightarrow \quad
  \tu_0(x) = \sum_{m=0}^\infty (c_m/\lambda_m) x(1-x)p_m(x).
\end{equation}
The simplest way to obtain an approximation of $\tf(x)$ in $\cP_{N-1}$
is to evaluate $f^{(1)}(x)$ at the zeros $x_j$ of $p_N(x)$, and to
compute the coefficients using Gauss-Lobatto quadrature:
\begin{equation}
  c_m = \int_0^1 \tf(x)p_m(x)x(1-x)\,dx = \int_0^1 f^{(1)}(x)p_m(x)\,dx \approx
  \sum_{j=1}^{N}f^{(1)}(x_j)p_m(x_j)\omega_j.
\end{equation}
Here we used the fact that $f^{(1)}(x)$ vanishes at $x_0=0$ and
$x_{N+1}=1$.  The abscissas $x_j$ and weights $\omega_j$ are easily
found using a variant of the Golub-Welsch algorithm
\cite{gautschi:book,golub:welsch}.  Further details of our numerical
implementation will be given elsewhere.

We can estimate the size of the coefficients $\{c_m\}$ in terms of the
  smoothness of $\fone,$ using the facts that
  \begin{equation}
    \begin{split}
      c_m&=\int_0^1 f^{(1)}(x)p_m(x)\,dx \qquad\text{and} \\
      \int_0^1\LKs{\btwo}p_m\fone dx&=\int_0^1p_m[\LK+2]\fone dx.
    \end{split}
  \end{equation}
The second formula is a special case of~\eqref{eq:shimakura} below. If $\fone$
is in $\cC^{2l}([0,1]),$ then we can iterate the integration by parts formula to obtain:
\begin{equation}
  c_m=\frac{(-1)^l}{[m(m+3)]^l}\int_0^1p_m(x)[\LK+2]^l\fone dx.
\end{equation}
As we show below,   there is a constant $C,$
independent of $m$ so that
\begin{equation}
  \|p_m\|_{L^{\infty}}\leq C\sqrt{m(m+3)}\|p_m\|_{L^2([0,1];d\mu_{\btwo})}=C\sqrt{m(m+3)∂}|,
\end{equation}
and therefore, there is a constant $\tC$ so that
\begin{equation}\label{eqn2.12}
  |c_m|\leq \tC\frac{\|[\LK+2]^l\fone\|_{L^{\infty}}}{[m(m+3)]^{l-\frac 12}}.
\end{equation}
In this instance
it is easy to see that
\begin{equation}
  \|[\LK+2]^l\fone\|_{L^{\infty}}\leq C_l\|[f\|_{\cC^{2l}}.
\end{equation}
This is almost a spectral estimate, but we lose one order of decay, in part
because we are estimating in the $L^{\infty}$-norm, and in part because there
is an implicit division of $\fone$ by $x(1-x)$ in the formula for $c_m.$

We can also use the $L^2$-norm to estimate these coefficients via
\begin{equation}\label{eqn2.12.1}
  |c_m|\leq \tC'\frac{\|w_{\btwo}^{-1}[\LK+2]^l\fone\|_{L^2([0,1];d\mu_{\btwo})}}{[m(m+3)]^{l}}.
\end{equation}
While the denominator is now $[m(m+3)]^{l}$ the numerator implicitly involves
the $L^2$-derivative of $|[\LK+2]^l\fone|^2$ at the boundary of $[0,1].$

Remarkably, very similar approaches work to find the eigenfunctions of $\LK$
and solve the Dirichlet problem in any dimension. This is
explained in the following two sections. In  Section~\ref{sec3} we give a
novel construction for the eigenfunctions of the neutral Kimura diffusion on
the $n$-simplex, which highlights their vanishing properties on subsets of the
boundary. In Section~\ref{sec4} we show how to solve the Dirichlet problem on
an $n$-simplex, with arbitrary smooth data.

\section{The Polynomial Eigenfunctions of $\LK$}\label{sec3}

In this section we give a hierarchical method of constructing the eigenfunctions of
$\LK,$ with considerable control over their vanishing properties on
$b\Sigma_n.$ As before, we let $\cP_d$ denote polynomials of
degree at most $d;$ the variables involved will be clear from the
context.

Our results are based upon a formula, which follows easily from a similar
calculation in the work of Shimakura, see Section 7 of~\cite{shimakura1}. As
noted above, we let $\LKs{\bb}=\LK +V_{\bb}$, where $V_{\bb}$ has linear coefficients
and assigns constant weights to the hypersurface components of $b\Sigma_n.$ We
then let $\cI=\{i_1<i_2<\cdots <i_{k+1}\}\subset\{1,\dots,n+1\}.$ Shimakura's
work implies the following formula
\begin{equation}\label{eq:shimakura}
  \LKs{\bb} (w_{\cI}\psi)=w_{\cI}\left[\LKs{\bb^{\prime}}-\kappa_{\cI}\right]\psi.
\end{equation}
Where
\begin{equation}
  w_{\cI}(x)=\prod_{j=1}^{k+1}x^{1-b_{i_j}}_{i_j};
\end{equation}
\begin{equation}
  b'_j=\begin{cases}
  2-b_j&\text{ if }j\in\{i_1,\dots,i_{k+1}\},\\
  b_j&\text{ if }j\notin\{i_1,\dots,i_{k+1}\};
\end{cases}
\end{equation}
and
\begin{equation}
  \kappa_{\cI}=
  \left(\sum_{j\in\cI}(1-b_j)\right)
  \left(k+\sum_{j\notin\cI}b_j\right).
\end{equation}
If $k=n$ and $\bb=\bzero,$ then $\bb'=\btwo=(2,\dots,2).$
Equation~\eqref{eq:shimakura} has a wide range of applications, and is
especially useful in cases where some of the $\{b_i\}$ are zero.  This formula gives a
very potent method to construct eigenfunctions that have a simple form
and vanish on certain parts of the boundary.

Recall that the $\cC^0$-graph closure of $\LK$ acting on
$\cC^3(\Sigma_n)$ is what is called the ``regular operator''
in~\cite{EpMaz2}, which is the ``backward'' Kolmogorov operator in
applications to Population Genetics. The eigenfunctions that we
construct are polynomials and hence in the domain of the regular
operator. Indeed, we will show that for each natural number $d,$ the
eigenfunctions of degree less than or equal to $d$ actually span
$\cP_d.$ Hence this is the complete set of eigenfunctions for the
regular operator.

As functions on the $n$-simplex, these eigenfunctions do not belong to a single
$L^2$-space, but each belongs to an $L^2$-space of some stratum of the
boundary. Their coefficients in the representation of a function in terms of
this spanning set are computed as inner products on these lower dimensional strata.

\subsection{Hierarchy of Polynomial Eigenfunctions}
\label{sec:efun}

The simplest polynomial eigenfunctions of $\LK$ are the functions
$\{x_1,\dots,x_{n+1}\},$ which are null vectors.  Each of these
eigenfunctions vanishes on a codimension 1 boundary face of
$\Sigma_n.$ Of course, a constant function is also a null-vector for
$\LK,$ but it already belongs to the span of the others since
$x_1+\cdots+x_{n+1}=1.$  To fit with the pattern below, one can
  write these functions as $x_{i_1}\psi(x)$, where $\psi(x)\equiv1$ spans
  the space of constant functions determined by their value at vertex
  $i_1$ (where $x_{i_1}=1$).

Next, for each distinct pair $1\leq i_1<i_2\leq n+1,$ we look
for eigenfunctions of the form $x_{i_1}x_{i_{2}}\psi(x).$ If
$\psi(x)$ is only a function of $x_{i_1}$ (or $x_{i_2}$) and
satisfies $[\LKs{\btwo}-2]\psi=\lambda\psi$, i.e.
\begin{equation}\label{eq:xi1:xi2:psi}
  x_{i_1}(1-x_{i_1})\pa^2_{x_{i_1}}\psi+2(1-2x_{i_1})
  \pa_{x_{i_1}}\psi-2\psi=\lambda\psi,
\end{equation}
then, by (\ref{eq:shimakura}), $x_{i_1}x_{i_{2}}\psi(x_{i_1})$ is
also an eigenfunction of the original operator $\LK $ with the same
eigenvalue.  Note that if
$\{j_1,\dots,j_{n-1}\}=\{1,\dots,n+1\}\setminus \{i_1,i_2\},$ then we
are solving on the edge
\begin{equation}
  x_{j_1}=\cdots=x_{j_{n-1}}=0.
\end{equation}
Hence the eigenfunctions $x_{i_1}x_{i_{2}}\psi(x_{i_1})$ vanish on the boundary of
this edge.  Also note that if a function $\psi$ depends only on a subset of the
coordinates, then $\LKs{\bb}\psi$ also depends only on the same subset of the
coordinates.  In particular, (\ref{eq:eig:1d})  and (\ref{eq:xi1:xi2:psi})
  agree.

With these observations, working in the different projective models,
we can construct all the polynomial eigenfunctions of $\LK.$ These
take the form $$x_{i_1}\cdots
x_{i_{k+1}}\psi(x_{i_1},\dots,x_{i_k}),$$ for various choices of
indices $\{i_1,\dots,i_{k+1}\}.$  Here $\psi$ is an eigenfunction
  of the operator
\begin{equation}
  \LKs{\cI,\bb^{\prime}} =\LKs{\bb^{\prime}}\restrictedto_{K_{\cI}}
\end{equation}
acting on a $k$-simplex, and the variables
$(x_{i_1},\dots,x_{i_{k+1}})$ range over the face defined by the
equations
\begin{equation}
  x_{j_1}=\cdots=x_{j_{n-k}}=0,
\end{equation}
using the notation of (\ref{eq:cJ}).
The weights for this operator are all equal to $2,$ and therefore
we denote it by $\LKs{\cI,\btwo}.$ A simple calculation shows that, if
$|\cI|=k+1,$ then
\begin{equation}\label{eq:efun:modP}
  \LKs{\cI,\btwo}\Big(x_{i_1}^{m_1}\cdots x_{i_k}^{m_k}\Big) =
    -\Big(|\vec m|^2+(2k+1)|\vec m|\Big)
    \Big(x_{i_1}^{m_1}\cdots x_{i_k}^{m_k}\Big)
    \mod \cP_{|\vec m|-1},
\end{equation}
where $|\vec m|=m_1 + \cdots + m_k$.  (A more general formula
  is given in (\ref{eq:efun:gen:modP}) below.) Thus,
$\LKs{\cI,\btwo}$ leaves the subspaces $\cP_d$ invariant; the $d$th
eigenvalue is $[-d^2-(2k+1)d]$ for $d\ge0$; and its multiplicity is
equal to ${d+k-1 \choose k-1}$, the dimension of
$\cP_d/\cP_{d-1}$. Moreover, applying the Gram-Schmidt procedure to
the set of monomials in the variables $x_{i_1}$, \dots, $x_{i_k}$,
ordered by degree (but arbitrarily ordered within the set of monomials
of the same degree), will lead to an orthonormal set of eigenfunctions
of $\LKs{\cI,\btwo}$. As a result, the eigenfunctions $\psi_{\cI,\vec
  m}(x_{i_1},\dots,x_{i_k})$ of this operator are multivariate
orthogonal polynomials with respect to the measure
\begin{equation}\label{eq:dmu2}
  d\mu_{\cI,\btwo}=\left(\textstyle\prod_{j=1}^{k+1} x_{i_j}\right)dx_{i_1}\cdots dx_{i_k}.
\end{equation}
The corresponding eigenfunctions $w_{\cI}\psi_{\cI,\vec m}$ of
  $\LKs{\cI}$ are orthogonal with respect to $d\mu_{\bzero}$ when
  $k=n$ and $\cI=\{1,\dots,n+1\}$, but are not normalizable (i.e.~do
  not belong to $L^2\big(\Sigma_n;d\mu_{\bzero}\big)$) when
  $k<n$. Nevertheless, they are still eigenfunctions algebraically,
  with eigenvalue

\begin{equation}
  \lambda_{\cI,\vec m} = -|\vec m|^2 - (2k+1)|\vec m| - k(k+1), \qquad
  (k=|\cI|-1)
\end{equation}
and play an essential role in solving $\LK u=f$ below,
where we use them to adjust $f$ to zero on the boundaries, just as was
done in (\ref{eq:f0:1d}) and (\ref{eq:u0:1d}) in 1-d.

From this observation, it is not difficult to demonstrate the completeness,
{as a Schauder basis for $\cC^0(\Sigma_n),$} of the eigenfunctions
obtained in this way. Let
$\cI=\{i_1,\dots,i_{k+1}\}\subset\{1,\dots,n+1\}$ be a set of indices, and
$K_{\cI}\subset b\Sigma_n,$ the corresponding boundary face. We let
$\dcP_{\cI}\subset\bbC[x_{i_1},\dots,x_{i_k}]$ denote the ideal of polynomials
defined on $K_{\cI}$ that vanish on $bK_{\cI}.$ It is easy to see that
\begin{equation}
  \dcP_{\cI}=x_{i_1}\cdots x_{i_{k+1}}\cdot\bbC[x_{i_1},\dots,x_{i_k}].
\end{equation}
Formula~\eqref{eq:shimakura} shows that this ideal is invariant under $\LKs{\cI}.$
The operator $\LKs{\cI,\btwo}$ is self adjoint with respect to the measure
$x_{i_1}\cdots x_{i_{k+1}}dx_{i_1}\cdots dx_{i_k},$ and it too preserves
$\cP_d,$ for any $d.$ Thus the polynomial eigenfunctions of $\LKs{\cI,\btwo}$
span $\bbC[x_{i_1},\dots,x_{i_k}],$ and therefore $x_{i_1}\cdots x_{i_{k+1}}$
times these eigenfunctions spans $\dcP_{\cI}.$

We next observe that an eigenfunction of this type must vanish
on every other $k$-simplex in the boundary of $\Sigma_n.$ This is
because one of the functions $\{x_{i_j}:\:j=1,\dots,k+1\}$ must appear
as a defining function for any other $k$-simplex. Using this
observation, we see that any polynomial $f$ that vanishes on
  every $k$-simplex except $K_{\cI}$ can be expanded in these
  eigenfunctions. This suggests a simple recursive method for finding
the regular solution of an equation of the form
\begin{equation}
  \LK u=f.
\end{equation}
As discussed below, the regular solution is required to belong to an
anisotropic H\"older space in the closed simplex rather than to take on
specified boundary values. Even if $f$ is a polynomial and $g=0,$ the solution to
the Dirichlet problem~\eqref{eqn1.3.01} is generally not differentiable up to
the boundary of $\Sigma_n.$

\subsection{The Regular Solution of $\LK u=f$}
As suggested by our construction of the eigenfunctions, the regular solution to
$\LK u=f$ is found recursively by working one boundary stratum at a time. 
The $k$-skeleton of $\Sigma_n$ is defined to be the union of $k$-simplices in
$b\Sigma_n.$ We denote this subset by $\Sigma^k_n.$ It is connected for $k>0,$ but
$\Sigma^k_n\setminus \Sigma^{k-1}_n$ is a disjoint union of \emph{open}
$k$-simplices with boundaries contained in $\Sigma^{k-1}_n.$ The regular
solution takes the form
\begin{equation}
  u=u_1+\dots+u_n,
\end{equation}
where the term $u_k$ is found by using the eigenfunctions constructed
above to solve {an auxiliary} inhomogeneous Dirichlet problem on
$\Sigma_n^k\setminus\Sigma^{k-1}_n.$ 

If $u\in\cC^2,$ then $\LK u$ vanishes at each vertex.  Hence, we start by
assuming that $f$ is a polynomial that vanishes on each of the vertices of
$\Sigma_n.$ This assumption is dropped in Section~\ref{sec4}, where
  imposing inhomogeneous Dirichlet boundary conditions {on $b\Sigma_n$} inevitably
  introduces singularities
  anyway.  We define $u_1$ as the solution of the equation $\LK
u_1\restrictedto_{\Sigma^1_n}=f\restrictedto_{\Sigma^1_n}$. As $f$ vanishes on
the vertices, which are the boundaries of the components of
$\Sigma_n^1\setminus\Sigma_n^0,$ we can use the eigenfunctions
$x_{i_1}x_{i_2}\psi(x_{i_1})$ constructed above to solve this equation on each
component of $\Sigma^1_n\setminus\Sigma^0_n$ independently of the others.  Note
that, using the eigenfunction representation, the function $u_1$ extends
canonically to the entire $n$-simplex.

We next solve
\begin{equation}
  \LK u_2\restrictedto_{\Sigma^2_n}=f\restrictedto_{\Sigma^2_n}-
  \LK u_1\restrictedto_{\Sigma^2_n}.
\end{equation}
The right hand side vanishes on $\Sigma^1_n,$ so we can use the
  eigenfunctions $w_{\cI}\psi_{\cI,\vec m}$ with
  $\cI=\{i_1,i_2,i_3\}$ to independently solve this equation on each
connected component of $\Sigma^2_n\setminus\Sigma^1_n.$ Recursively,
we assume that we have found $u_1,\dots,u_{k-1},$ so that
\begin{equation}\label{eq:f:stage:k}
  f-\LK (u_1+\cdots+u_{k-1})
\end{equation}
vanishes on the $(k-1)$-skeleton, and then solve
\begin{equation}
  \LK u_k\restrictedto_{\Sigma^k_n}=f\restrictedto_{\Sigma^k_n} -
  \LK (u_1+\cdots+u_{k-1})\restrictedto_{\Sigma^k_n}.
\end{equation}
Using the eigenfunctions of the form $x_{i_1}\cdots
x_{i_{k+1}}\psi(x_{i_1},\dots,x_{i_k}),$ we can solve the relevant
Dirichlet problems independently on each component of
$\Sigma^k_n\setminus\Sigma^{k-1}_n.$ 

The process terminates when we reach the interior of the $n$-simplex, where we
solve the problem
\begin{equation}
  \LK u_{n}\restrictedto_{\Sigma_n}=f\restrictedto_{\Sigma_n} -
  \LK (u_1+\cdots+u_{n-1})\restrictedto_{\Sigma_n}.
\end{equation}
Once again the right hand side vanishes on the entire boundary of
$\Sigma_n$ and we can solve this equation using eigenfunctions of the
form $x_1\cdots x_{n+1}\psi(x).$ Since this can be done for any
polynomial that vanishes on the vertices, it demonstrates that
the eigenfunctions constructed above, including the
  nullspace, are in fact a complete set of eigenfunctions for the
  operator $\LK,$ acting on polynomial functions defined on
  $\Sigma_n.$ Since these functions are dense in $\cC^2(\Sigma_n)$ it
  follows easily that this is also a complete set of eigenfunctions
  for the graph closure of $\LK$ acting on $\cC^0(\Sigma_n).$
Altogether we have proved the following result:
\begin{theorem}
  The regular operator $\LK$ acting on functions defined on $\Sigma_n$ has a complete
  set of eigenfunctions of the form
  \begin{equation}
    \cE(\LK)=\{x_1,\dots,x_{n+1}\}\cup
    \bigcup_{k=1}^{n}\bigcup_{\cI=\{1\leq
      i_1<\cdots<i_{k+1}\leq n+1\}}x_{i_1}\cdots x_{i_{k+1}}\cE(\LKs{\cI,\btwo}).
  \end{equation}
Here $\cE(\LKs{\cI,\btwo})$ denotes the eigenfunctions of the operator
$\LKs{\cI,\btwo},$ which are polynomials in the variables $\{x_{i_1},\dots,x_{i_k}\}.$
\end{theorem}
\begin{remark} Another useful consequence of the hierarchical description of
  the eigenfunctions of $\LK$ is that it dramatically reduces the work needed
  to find the eigenfunctions of $\LK$ acting on $\Sigma_{n+1},$ once they are
  known for $\Sigma_{n}.$ Indeed, up to choosing subsets of coordinates
  $(x_1,\dots,x_{n+1}),$ all that is required is to find the ``new'' eigenfunctions,
  which are of the form $$x_1\cdots x_{n+2}\psi(x_1,\dots, x_{n+1}).$$
\end{remark}

\begin{remark}
At stage $k$ of the recursive algorithm above, we are given a
function $f$ that vanishes on the $(k-1)$-skeleton of $\Sigma_n$ and wish to
find $u_k$ such that $\LK u_k\restrictedto_{\Sigma^k_n} =
f\restrictedto_{\Sigma^k_n}$. For simplicity of notation, we have absorbed
$\LK(u_1+\cdots u_{k-1})$ into $f$ in (\ref{eq:f:stage:k}).  Since the faces of
the $k$-skeleton decouple, we may assume $f=0$ on all faces $K_\cI$ except one
of them, which we take to be $\cI=\{1,\dots,k+1\}$. Recall from the discussion
after (\ref{eq:efun:modP}) that the relevant eigenfunctions of $\LKs\cI$ are of
the form $(x_1\cdots x_{k+1})\psi_{\cI,\vec m}(x)$, with $\{\psi_{\cI,\vec
  m}(x)\}_{\vec m\in \mathbb{N}_0^k}$ a set of multivariate orthogonal
polynomials with respect to $d\mu_{\cI,\btwo}$ in (\ref{eq:dmu2}) on
$\tSigma_k$. Thus, the generalization of (\ref{eq:expansion:1d}) to $k$
dimensions is
\begin{equation}\label{eq:expansion:kd}
  \tf(x) = \frac{f(x)}{x_1\cdots x_{k+1}} = \sum_{\vec m\in\mathbb{N}_0^k}
  c_{\vec m} \psi_{\cI,\vec m}(x) \quad \Rightarrow \quad
  u_k(x) = \sum_{\vec m\in\mathbb{N}_0^k}
  \frac{c_{\vec m}}{\lambda_{\vec m}} x_1\cdots x_{k+1}\,\psi_{\cI,\vec m}(x).
\end{equation}
\end{remark}

In~\cite{EpMaz2} it is shown that for $k$ a non-negative integer, and
$\gamma\in (0,1),$ the operator $\LK $ maps the anisotropic H\"older space
$\cC^{k,2+\gamma}_{\WF}(\Sigma_n)$ to  $\cC^{k,\gamma}_{\WF}(\Sigma_n).$ If we let
$\dcC^{*,*}_{\WF}$ denote the closed subspaces of functions vanishing at the
vertices of $\Sigma_n,$ then the inverse
\begin{equation}
 (\LK )^{-1} :\dcC^{k,\gamma}_{\WF}(\Sigma_n)\longrightarrow\dcC^{k,2+\gamma}_{\WF}(\Sigma_n)
\end{equation}
is shown to be a bounded operator. Let $C_{k,\gamma}$ denote the bound on this
operator.  If we approximate $f\in\dcC^{k,\gamma}_{\WF}(\Sigma_n)$ by
a polynomial, $p,$ then
\begin{equation}
  \|(\LK )^{-1}f-(\LK )^{-1}p\|_{\WF,k,2+\gamma}\leq C_{0,\gamma}\|f-p\|_{\WF,k,\gamma}.
\end{equation}
The classical $(k,\gamma)$-H\"older norms dominate the WF-H\"older
  norms, and $\cC^{k,\gamma}_{\WF}\subset\cC^k.$ It is well known that the
  accuracy of the best degree $d$ polynomial approximant only depends on the
  classical H\"older smoothness of the data, thus demonstrating the efficacy of
  our method for solving the equation $\LK u=f$ for any reasonably smooth data
  $f$ that vanishes at the vertices of $\Sigma_n.$

  One can generalize the construction of polynomial eigenfunctions to any
  Kimura operator, $\LKs{\bb},$ with constant weights.
%
%
  If all of the weights are positive, the measure $d\mu_{\bb}$ has
  finite mass and the polynomial eigenfunctions of the regular
  operator can be chosen to be orthonormal with respect to this
  measure.  This case is, in many ways, the easiest to deal with as
  there is a single ambient Hilbert space containing all the
  eigenfunctions.  It has been studied previously by solving
    hypergeometric equations \cite{kimura1956,karlin:64}, by computing
    bi-orthogonal systems of polynomials \cite{karlin:64,littler1975},
    or by computing orthogonal polynomials on weighted Hilbert spaces
    and grouping them into reproducing kernels
    \cite{griffiths1979,griffiths2010,griffiths2011}. The present work
    extends this third approach to the case of 0 weights, and, for
    reasons explained below, does not
    take the final step of forming the reproducing kernels.  In the
  construction of Section~\ref{sec:efun} above, we look for
  eigenfunctions of the form $x_{i_1}\cdots
  x_{i_{k+1}}\psi(x_{i_1},\dots,x_{i_k})$ with $\psi$ an
  eigenfunction of the auxiliary operator $\LKs{\bb^{\prime}}$.  If
  all the weights are positive, the leading factor of $x_{i_1}\cdots
  x_{i_{k+1}}$ is dropped, $k$ is set to $n$, and the eigenfunctions
  of $\LKs{\bb}$ are computed directly, without recourse to an
  auxiliary operator.  While eigenfunctions with respect to a
    subset of the variables are again eigenfunctions of $\LKs{\bb}$,
    they are already present in the construction at level $n$, and do
    not have to be dealt with recursively along the stratification of
    $\pa\Sigma_n.$ Computing these eigenfunctions works the same in
  both cases, and is mathematically equivalent to applying the
  Gram-Schmidt method to the monomials ordered by total degree. This
  is explained in the following section.

 When some weights are zero and some positive,
  Shimakura's formula leads to a partial hierarchical structure for the
  eigenfunctions, similar to that described here. This is discussed, to some
  extent, in~\cite{shimakura1,shimakura2}. As shown in~\cite{shimakura1} it is
  possible to specify homogeneous Dirichlet data on any boundary face where the
  weight $b_i<1.$ In general the solution is not smooth along this face, but
  has a leading singularity of the form $x_i^{1-b_i}.$ We will return to these
  questions in a subsequent publication.

\section{Numerical Construction of Eigenfunctions}
\label{sec:numerics}
We now give a practical method to construct the eigenfunctions described
above. By a change of variables these eigenfunctions can be represented as
products of functions of single variables, which are themselves eigenfunctions
of differential operators. Since it is no more difficult, we consider the
general case of a Kimura diffusion on $\Sigma_n$ with constant weights.

In 2-d, it was observed by Proriol \cite{proriol:57} that orthogonal
polynomials on a triangle can be represented as tensor products of 1-d
Jacobi polynomials.  See Koornwinder \cite{koornwinder:75} for further
background and more complicated 2-d geometries. A rather different
approach using hypergeometric functions was developed in 2-d  by
  Kimura \cite{kimura1956} and Karlin and McGregor \cite{karlin:64}.
This approach was extended to the simplex or $n$-sphere in the
context of the Laplace-Beltrami equation by Kalnins, Miller and
Tratnik \cite{kalnins:91}.  Littler and Fackerell
  \cite{littler1975} found solutions of Kimura diffusion using
  expansions in bi-orthogonal polynomials.  Griffiths
  \cite{griffiths1979} found representations for the transition
  function for this diffusion process using reproducing kernels
  expressed in terms of multivariate Jacobi polynomials; see also
  \cite{griffiths2010,griffiths2011}.  The method we derive below is
  equivalent to Griffiths' approach, though we avoid computing
  reproducing kernels as it is more efficient and numerically stable
  (at the expense of symmetry) to leave the eigenfunctions separated.
Wingate and Taylor \cite{wingate:98} showed how to generalize the
Proriol construction to the $N$-dimensional simplex in the case of a
uniform weight.  We adapt their method to allow more general Jacobi
weights of the form
\begin{equation}\label{eq:inner}
  \la p,q\ra = \int_{\tSigma_k} p(x)q(x)\,
  x_1^{\alpha_1}\cdots x_{k+1}^{\alpha_{k+1}}\,
  dx_1\cdots dx_k.
\end{equation}
The case of interest here is $\alpha_j=b_j'-1=1$, but treating the general
case is no more complicated, and can be used to study Kimura diffusion
with non-zero weights $\bb$, as shown below.  

The idea is to map the unit cube to the simplex in such a way that the desired
orthogonal polynomials on the simplex pull back to have tensor product form on
the cube. One way to do this, which differs somewhat from the choice made by
Wingate and Taylor, is with the change of variables $x=T(X)$ given implicitly
by
\begin{equation}\label{eq:change:vars}
  x_1=X_1, \quad x_2 = (1-x_1)X_2, \quad \cdots \quad x_k = (1-x_1-\cdots-x_{k-1})X_k,
\end{equation}
which solves to $x_j=\big(\prod_{i=1}^{j-1}(1-X_i)\big)X_j$ and
$x_{k+1}=1-x_1-\cdots-x_k=\prod_{i=1}^k(1-X_i)$. These are blow-ups similar to
the changes of variables that are used in~\cite{Jost1, Jost2, Jost3}.

The Jacobian matrix $DT(X)$ is lower triangular, so its determinant is easy to
compute:
\begin{equation}
  J=|\det DT|=\textstyle \prod_{j=1}^k (1-X_j)^{k-j}.
\end{equation}
The inner product (\ref{eq:inner}) may now be written
\begin{equation}\label{eq:inner2}
  \la p,q \ra = \int_{[0,1]^k} (p\circ T)(q\circ T)
  \Big(\prod_{j=1}^k\big[ X_j^{\alpha_j}
    (1-X_j)^{(k-j)+\sum_{i=j+1}^{k+1}\alpha_i}\big]\Big)\,dX_1\cdots dX_k.
\end{equation}
Next we note that if $p$ is a polynomial (in one variable) of degree
$d$, then
\begin{equation}
  p(X_j)\prod_{i=1}^{j-1}(1-X_i)^d = p\left(
  \frac{x_j}{1-x_1-\cdots-x_{j-1}}\right)(1-x_1-\cdots-x_{j-1})^d
\end{equation}
is a polynomial of degree $d$ in the variables $x_1,\cdots,x_j$.
We can therefore define $\psi_{\vec m}(x)$ via
\begin{equation}\label{eq:p:vec:m:def}
\begin{aligned}
  \psi_{\vec m}\circ T(X) &= \prod_{j=1}^k \left(
  Q^{(\alpha_j,\alpha_{j\vec m})}_{m_j}(X_j)
  \prod_{i=1}^{j-1}(1-X_i)^{m_j}\right) \\
  &= \prod_{j=1}^k\left(Q^{(\alpha_j,\alpha_{j\vec m})}_{m_j}(X_j)(1-X_j)^{\sum_{i=j+1}^k m_i}
  \right),
  \end{aligned}
\end{equation}
where
\begin{equation}
  \alpha_{j\vec m} = \alpha_{k+1}+\sum_{i=j+1}^k(\alpha_i+2m_i+1).
\end{equation}
Here $\{Q_m^{(\alpha,\beta)}\}_{m=0}^\infty$ are
orthogonal polynomials on $[0,1]$ in the inner product
(\ref{eq:inner:prod}), normalized to have unit length.
Substitution into (\ref{eq:inner2}) gives
\begin{multline*}
  \la \psi_{\vec m},\psi_{\vec m'} \ra =\\
  \prod_{j=1}^k \int_0^1 Q^{(\alpha_j,\alpha_{j\vec m})}_{m_j}(X_j)
  Q^{(\alpha_j,\alpha_{j\vec m'})}_{m_j'}(X_j)
  X_j^{\alpha_j} (1-X_j)^{\alpha_{k+1}+\sum_{i=j+1}^k(\alpha_i+m_i+m_i'+1)}\,dX_j,
\end{multline*}
which, proceeding from $j=k$ to $j=1$, may be seen to equal
$\prod_{j=1}^k \delta_{m_jm_j'}$.  We note that $\psi_{\vec m}(x)$ is a
polynomial of degree $d=|\vec m|=\sum_1^k m_j$.  We order the
multi-indices first by degree ($\vec m<\vec m'$ if $d<d'$) and then
lexicographically from right to left (if $d=d'$ then $\vec m<\vec m'$
if $m_k<m'_k$, or if $m_k=m'_k$ and $m_{k-1}<m'_{k-1}$, etc., so that
$(d,0,\dots,0)$ is smallest).  In this ordering, the $\psi_{\vec m}(x)$
are the same as one would get from applying Gram-Schmidt to the
monomials $x^{\vec m}$ in the same order.  Hence, all polynomials
in $k$ variables are accounted for by this construction. For example,
when $k=3$ and $\alpha_1=\alpha_2=\alpha_3=\alpha_4=1$, the first
10 orthogonal polynomials are
\small
\begin{align*}
  &\psi_{000}=12\sqrt{35}, \quad
  \psi_{100}=12\sqrt{105}(4x-1), \quad
  \psi_{010}=12\sqrt{210}(x+3y-1), \\
  &\psi_{001}= 36 \sqrt{70} (-1 + x + y + 2 z), \quad
  \psi_{200}= 24 \sqrt{55} (1 - 9 x + 15 x^2), \\
  &\psi_{110}= 6 \sqrt{2310} (-1 + 5 x) (-1 + x + 3 y), \quad
  \psi_{020}= 6 \sqrt{330} (3 + 3 x^2 - 21 y + 28 y^2 + 3 x (-2 + 7 y)), \\
  &\psi_{101}= 18 \sqrt{770} (-1 + 5 x) (-1 + x + y + 2 z), \quad
  \psi_{011}= 30 \sqrt{462} (-1 + x + 4 y) (-1 + x + y + 2 z), \\
  &\psi_{002}= 24 \sqrt{1155} (1 + x^2 + y^2 - 5 z + 5 z^2 + y (-2 + 5 z) + 
   x (-2 + 2 y + 5 z)),
\end{align*}
\normalsize
which span the same space as $1,x,y,z,x^2,xy,y^2,xz,yz,z^2$ and are
pairwise orthogonal. In numerical computations, it is best to leave them
in the form (\ref{eq:p:vec:m:def}) since representation in the
monomial basis is both expensive and numerically unstable when the
degree is large.

The $Q_m^{(\alpha,\beta)}$ are proportional to Jacobi polynomials,
$Q_m^{(\alpha,\beta)}(x)\propto P_m^{(\beta,\alpha)}(2x-1)$, but are
more easily computed directly from a 3-term recurrence.
Explicitly, $p_m(x)=Q_m^{(\alpha,\beta)}(x)$
satisfies the recurrence (\ref{eq:recur}) with $\gamma_0 =
\frac{\Gamma(\alpha+1)\Gamma(\beta+1)}{\Gamma(\alpha+\beta+2)}$ and
\begin{alignat*}{2}
  a_m &= \frac{1}{2}\left(1 + \frac{(\alpha+\beta)(\alpha-\beta)}{(\alpha+\beta+2m+2)(\alpha+\beta+2m)}\right)
  \overset{m=0}{=} \frac{\alpha+1}{\alpha+\beta+2}, & \qquad & (m\ge0), \\
  b_m &= \frac{m(\alpha+m)(\beta+m)(\alpha+\beta+m)}{(\alpha+\beta+2m)^2[(\alpha+\beta+2m)^2-1]}
  \overset{m=1}{=}\frac{(\alpha+1)(\beta+1)}{(\alpha+\beta+2)^2(\alpha+\beta+3)}, &  &(m\ge1).
\end{alignat*}
These formulas can be derived from the well-known recurrence relations
of the Jacobi polynomials \cite{gautschi:book,abramowitz}, modified so that
$\|p_m\|=1$.  Finally,
now that we have specified the orthonormal basis $\psi_{\vec m}(x)$, the
coefficients
\begin{equation}
  c_{\vec m} = \int_{\Sigma_k} \tf(x)\psi_{\vec m}(x) x_1^{\alpha_1}\cdots x_{k+1}^{\alpha_{k+1}}\,
  dx_1\cdots dx_k = \int_{\Sigma_k} f(x)\psi_{\vec m}(x)\,dx_1\cdots dx_k
\end{equation}
in (\ref{eq:expansion:kd}) can be computed by Gauss-Jacobi quadrature
on $[0,1]^k$ via the same change of variables
(\ref{eq:change:vars}). Further details will be presented elsewhere.

To see that the orthogonal polynomials $\psi_{\vec m}$ constructed above
are eigenfunctions of $\LKs{\cI,\bb}$, we note that
the generalization of (\ref{eq:efun:modP}) is
\begin{equation}\label{eq:efun:gen:modP}
  \LKs{\cI,\bb}\Big(x_1^{m_1}\cdots x_{k+1}^{m_{k+1}}\Big) =
    -\Big(|\vec m|^2+(B-1)|\vec m|\Big)
    \Big(x_1^{m_1}\cdots x_{k+1}^{m_{k+1}}\Big)
    \mod \cP_{|\vec m|-1},
\end{equation}
where $B=b_1+\cdots+b_{k+1}$ and $b_j=\alpha_j+1$.
Normally $m_{k+1}=0$ as the simplex is parametrized by any $k$ of
the variables $x_i$. Equation (\ref{eq:efun:gen:modP}) shows that
$\LKs{\cI,\bb}$ is upper triangular in any monomial basis ordered by
degree (with arbitrary ordering among monomials of the same degree),
and the eigenvalues can be read off the diagonal. If the Gram-Schmidt
procedure is applied to these monomials using the $d\mu_{\bb}$ inner
product, the resulting system of orthogonal polynomials will span a
nested sequence of invariant subspaces for $\LKs{\cI,\bb}$, and hence
are eigenfunctions, by self-adjointness. The above construction is
equivalent to the Gram-Schmidt procedure.

Following Griffiths \cite{griffiths1979}, one could combine all the
eigenfunctions in each eigenspace into reproducing kernels, defined as
$G_d(x,y) = \sum_{|\vec m|=d} \psi_{\vec m}(x)\psi_{\vec m}(y)$.  In
the above example, replacing $(x,y,z)$ with $(x_1,x_2,x_3)$, one would
have e.g.~$G_1(x,y) = 136080 + 181440[-x_1-x_2-x_3-y_1-y_2-y_3+
  x_1y_1+x_2y_2+x_3y_3+ (x_1+x_2+x_3)(y_1+y_2+y_3)]$. This is
theoretically appealing due to symmetry and independence of the
particular ordering chosen to compute the orthogonal
polynomials. However, it is more efficient computationally to work
with the orthogonal polynomial representation directly. It is also
numerically unstable to expand out these polynomials in terms of their
coefficients to achieve symmetry in the formulas for $G_d(x,y)$ when
$d$ is large.

\subsection{Coefficient Estimates}

If $f$ is a polynomial, or more generally, a smooth enough function on
$\Sigma_n,$ then we can use the recursive approach above to represent $f$ in
terms of the eigenfunctions of $\LK,$ $\{w_{\cI}\psi_{\cI,\vec m}\}.$ Here and
in the sequel we assume that all the eigenfunctions $\{\psi_{\cI,\vec m}\}$ are
normalized by the condition:
\begin{equation}
  \int\limits_{K_{\cI}}|\psi_{\cI,\vec m}(y)|^2d\mu_{\cI,\btwo}(y)=1.
\end{equation}
{When $\cI=\{j\},$ then $K_{\cI}$ is a vertex; the functions on $K_{\cI}$
  are constants.  We write $\psi_{\cI}$ instead of $\psi_{\cI,\vec m},$ where 
  for all $j,$ $\psi_{\{j\}}(x)=1$, the constant function. Note that
  $x_j\psi_{\{j\}}(x)$ is then a function in the null-space of $\LK$ that is
  one at $e_j$ and vanishes at all other vertices. We let
  $\mu_{\cI,\btwo}=\delta_{e_j}(x)$ be the atomic measure at the vertex $e_j$ of
  unit mass. }

We begin by subtracting off the values of $f$ at the vertices:
\begin{equation}
  \fone(x) = f(x) - \sum_{j=1}^{n+1}f(e_j) {x_j}\psi_{\{j\}}(x).
\end{equation}
The function $\fone$ now vanishes at
each of the vertices of $\Sigma_n$ and can therefore be expanded on the
1-skeleton in terms of the eigenfunctions of the form $\{w_{\cI}\psi_{\cI,\vec m}\},$
  where $\cI$ ranges over subsets of $\{1,\dots,n+1\}$ of cardinality $2.$ The
coefficients are computed as integrals over the 1-dimensional strata:
\begin{equation}
  c_{\cI,\vec m}=\int_{0}^1\frac{\fone(x)}{x(1-x)}\psi_{\cI,\vec m}(x)
  x(1-x)dx=\int_{0}^1\fone(x)\psi_{\cI,\vec m}(x)dx,
\end{equation}
where $x$ is a coordinate on the stratum $K_{\cI}.$ 
 As in~\eqref{eqn2.12} and~\eqref{eqn2.12.1}, the
rate of decay of these coefficients is determined by the smoothness of
$\fone.$

Subtracting off the contributions from the 1-dimensional strata we get
  $\ftwo,$ which vanishes on the 1-skeleton of $\Sigma_n.$ we can represent
  this function on the 2-skeleton in terms of eigenfunctions of the form
  $\{w_{\cI}\psi_{\cI,\vec m}\},$ where the cardinality of $\cI$ is $3.$
  Proceeding in this way we get a sequence of functions
  $f,\fone,\ftwo,\dots,\fk{k},\dots,\fk{n},$ where $\fk{k}$ vanishes on the
  $k-1$-skeleton of $b\Sigma_n.$

The coefficients coming from $\fk{k}$ are
  computed as integrals over the components of the $k$-dimensional part of the
  boundary of $\Sigma_n:$
  \begin{equation}
    c_{\cI,\vec m}=\int\limits_{K_{\cI}}\psi_{\cI,\vec m}(y)\fk{k}(y)dy.
  \end{equation}
  Here $\cI$ has cardinality $k+1$ and $y$ is a linear coordinate on $K_{\cI}.$
Using Shimakura's formula,~\eqref{eq:shimakura} we can show that
\begin{equation}
  \int\limits_{K_{\cI}}L_{\btwo}\psi_{\cI,\vec m}(y)\fk{k}(y)dy=
\int\limits_{K_{\cI}}\psi_{\cI,\vec m}(y)[\LKs{\cI,\bzero}+\kappa_{\cI}]\fk{k}(y)dy.
\end{equation}
From this we can show that if $L_{\btwo}\psi_{\cI,\vec m}=\lambda_{\cI,\vec
  m}\psi_{\cI,\vec m},$ then, for a function $\fk{k}\in\cC^{2l}(K_{\cI})$ that
vanishes at the boundary of $K_{\cI},$ we have:
\begin{equation}\label{eqn3.23}
    c_{\cI,\vec m}=
\frac{1}{\lambda_{\cI,\vec m}^l}\int\limits_{K_{\cI}}\psi_{\cI,\vec m}(y)[\LKs{\cI,\bzero}+\kappa_{\cI}]^l\fk{k}(y)dy.
  \end{equation}
From this relation it is clear that the rate of decay of the coefficients
$\{c_{\cI,\vec m}\}$ is determined by the smoothness of $\fk{k}$ on the
$k$-skeleton, and an $L^{\infty}$-estimate for the eigenfunctions
$\{\psi_{\cI,\vec m}\}.$

To estimate $\fk{k}$ in terms of the original data
would take us too far afield, so we conclude this discussion by proving 
sup-norm estimates on the eigenfunctions. The eigenfunctions $\{\psi_{\cI,\vec
  m}\}$ 
are eigenfunctions of the operator $L_{\cI,\btwo},$ which is self
adjoint with respect to the measure $d\mu_{\cI,\btwo},$ and has strictly positive
weights. The kernel for the operator $e^{tL_{\cI,\btwo}}$ takes the form
$p_t(y,\ty)d\mu_{\cI,\btwo}(\ty),$ with $p_t(y,\ty)=p_t(\ty,y).$ This and the
semi-group property easily imply
 that
\begin{equation}
  p_{2t}(y,y)=\int_{K_{\cI}}[p_{t}(y,\ty)]^2d\mu_{\cI,\btwo}(\ty).
\end{equation}
Since the operator has positive weights, we can use the Theorem 5.2
in~\cite{EpMaz4} to conclude that there is a constant $C_k$ depending only on
the dimension so that
\begin{equation}
  p_{2t}(y,y)\leq \frac{C_k}{\mu_{\btwo}(B_{\sqrt{2t}}(y))}.
\end{equation}
Here $B_{r}(y)$ is the ball in the intrinsic metric (see~\eqref{eqn1.23} and~\cite{EpMaz4}) of radius $\sqrt{2t}$
centered at $y.$ From the forms of the metric and the measure we can easily show that there
are constants $C_0, C_1,$ so that for small $t,$ on strata of dimension $k,$ we have  the bounds
\begin{equation}
  C_0t^{k}\leq \mu_{\btwo}(B_{\sqrt{2t}}(y))\leq C_1 t^{\frac k2}.
\end{equation}
To prove an estimate on $\|\psi_{\cI,\vec m}\|_{L^{\infty}},$ we observe that
\begin{equation}
  \psi_{\cI,\vec m}(y)e^{t\lambda_{\cI,\vec m}} =\int\limits_{K_{\cI}}p_t(y,\ty)\psi_{\cI,\vec m}(\ty)d\mu_{\cI,\btwo}(\ty).
\end{equation}
The Cauchy-Schwarz inequality and the estimates above show that
\begin{equation}
  |\psi_{\cI,\vec m}(y)|\leq \frac{C_ke^{-t\lambda_{\cI,\vec m}}}{t^{\frac k2}}.
\end{equation}
Setting $t=-1/\lambda_{\cI,\vec m},$ gives the estimate
\begin{equation}
  |\psi_{\cI,\vec m}(y)|\leq \tC_k|\lambda_{\cI,\vec m}|^{\frac k2}.
\end{equation}

Inserting this estimate into~\eqref{eqn3.23} we see that there is a constant
$C'_k,$ so that the coefficients coming from the stratum of dimension $k$
satisfy an estimate of the form:
\begin{equation}
  |c_{\cI,\vec m}|\leq
  C'_k\frac{\|[\LK+\kappa_{\cI}]^l\fk{k}\|_{L^{\infty}}}{\lambda_{\cI,\vec m}^{l-{\frac k2}}}.
\end{equation}
One-half order of decay is lost with each increase in dimension.  We can also
give an $L^2$-estimate, where no such loss explicitly occurs, wherein
\begin{equation}
  |c_{\cI,\vec m}|\leq
  C''_k\frac{\|w_{\cI}^{-1}[\LK+\kappa_{\cI}]^l\fk{k}\|_{L^{2}(K_{\cI},d\mu_{\cI,\btwo})}}{\lambda_{\cI,\vec m}^{l}}.
\end{equation}
These estimates implicitly involve $k$ derivatives of
$|[\LK+\kappa_{\cI}]^l\fk{k}|^2$ near the boundary of $K_{\cI}.$  In both estimates there are
further losses that occurs in the estimation of $[\LK+\kappa_{\cI}]^l\fk{k}$ in
terms of the original data $f.$

\section{The Dirichlet Problem}\label{sec4}

In many applications of the Kimura equation to problems in population genetics
one needs to solve a problem of the form
\begin{equation}\label{dirprob}
  \LK u=f\text{\; with\; }u\restrictedto_{S}=g.
\end{equation}
Here $S$ is a subset of $b\Sigma_n,$ generally assumed to be a union of faces.
In the constant weight case, Shimakura showed that this problem is well-posed
if the weights on the faces contained in $S$ are all less than $1.$ Note that,
if a weight is greater than or equal to 1, then, with probability 1, the paths
of the associated stochastic process never reach the corresponding face. 

In this section, for simplicity, we continue to consider the case that all
weights are zero, $S=b\Sigma_n,$ and that $f$ and $g$ have a certain degree of
smoothness. With this assumption, we show that the solution to the Dirichlet
problem has an asymptotic expansion along boundary with the first two terms
determined by local calculations, see~\eqref{eqn6.19} and~\eqref{eqn6.19.2}.  A
classical example from Population Genetics is the solution to the Dirichlet
problem with $f=-1,$ and $g=0,$ which is given by
\begin{equation}
  u(x)=\sum_{j=1}^n\sum_{0\leq i_1<\cdots<i_j\leq n}\eta(x_{i_1}+\cdots+x_{i_j})(-1)^{j},
\end{equation}
where $\eta(\tau)=\tau\log \tau.$ For a point $x\in\Sigma_n,$ the value $u(x)$ is the
expected time for a path starting at $x$ to reach $b\Sigma_n.$ This formula is
highly suggestive of the form that the general result takes;
see~\eqref{singprt0} and~\eqref{singprtsplx} below.

Boundary data that is not continuous on $b\Sigma_n$ does arise naturally in
problems connected with exit probabilities, and cannot be directly treated by
the methods in this paper. In these types of problems the boundary data is
often piecewise constant assuming only the values 0 and 1, and one has recourse
to explicit solutions, see~\cite{littler1975-1}. In this case, one could
approximate the discontinuous boundary data by smooth boundary data, and use
the methods presented here along with the Feynman-Kac formula
in~\cite{EpsteinPop} to get upper and lower bounds for the solutions of such
boundary value problems. For data without some degree of regularity one should
not expect the solution to have a simple, explicit asymptotic expansion along
the boundary. When the solution is not continuous up to the boundary,
considerable care is required in the interpretation of the partial differential
equation along the boundary, and boundary condition itself, see~\cite{Jost1,
  Jost2, Jost3}.

To start we consider the simpler problem in a positive orthant
$S_{n,0}=\bbR_+^n,$ with the model operator
\begin{equation}\label{modop0}
  L_{\bzero}=\sum_{i=1}^{n}x_i\pa_{x_i}^2.
\end{equation}
This problem is easier to solve and its solution is nearly adequate to solve
the analogous problem in a simplex. We start with a simple calculus lemma.
\begin{lemma}\label{lem1} Let $\psi$ be a $\cC^2$ function of a single
  variable, then for indices $1\leq i_1<\cdots<i_k\leq n,$ we have
  \begin{equation}
    L_{\bzero}\psi(x_{i_1}+\cdots +x_{i_k})=(x_{i_1}+\cdots +x_{i_k})\psi''(x_{i_1}+\cdots +x_{i_k}).
  \end{equation}
\end{lemma}
\noindent
The proof is an elementary calculation.

We now show how to solve the problem
\begin{equation}\label{dirprob0}
  L_{\bzero}u=f\text{ in }S_{n,0}\text{\; with\; }u\restrictedto_{bS_{n,0}}=g,
\end{equation}
for $f$ a compactly supported function in $\cC^2(S_{n,0}).$ We assume that $g$
has a compactly supported extension, $\tg,$ to $S_{n,0},$ for which
$L_{\bzero}\tg$ also belongs to $\cC^2(S_{n,0}).$ It is clear that, generally,
this problem cannot have a regular solution $u\in\cC^2(S_{n,0}).$ If $u$
belongs to this space, then, for any substratum $\sigma$ of $\pa S_{n,0},$ we
have
\begin{equation}
  (L_{\bzero}u)\restrictedto_{\sigma}=L_{\bzero}\restrictedto_{\sigma} u\restrictedto_{\sigma}.
\end{equation}
Hence  $f$ and $g$ would have to satisfy the very restrictive compatibility
conditions
\begin{equation}
  f\restrictedto_{\sigma}=L_{\bzero}\restrictedto_{\sigma} g\restrictedto_{\sigma}.
\end{equation}

We look for a solution of the form $u=\tg+v,$ with $v$ solving
\begin{equation}\label{dirprob00}
  L_{\bzero}v=f-L_{\bzero}\tg\overset{\text{def}}{=}\fone\text{\; and\; }
  v\restrictedto_{bS_{n,0}}=0.
\end{equation}
Our first goal is to give a method for solving~\eqref{dirprob00}, which gives a
precise description of the singularities that arise for general smooth,
compactly supported data $(f,g).$ We begin by giving a precise definition to
the meaning of the equations in~\eqref{dirprob0}: A function $u\in\cC^2(\Int
S_{n,0})\cap \cC^0(\oS_{n,0})$ solves~\eqref{dirprob0} if
$u\restrictedto_{bS_{n,0}}=g,$ and so that $L_{\bzero}u,$ which is
initially defined only in the interior of the orthant, has a continuous
extension to $\oS_{n,0}$ with $L_{\bzero}u=f,$ throughout.

In Section~\ref{sec2} we showed how the analogous problem for $\LK$ is solved
in 1-dimension. We begin this construction by explaining how to
solve~\eqref{dirprob0} when $n=1:$ For a fixed $\epsilon\ll1,$ we let
$\psi\in\cC^{\infty}_c([0,\epsilon))$ be a non-negative function, which is $1$
in $[0,\frac{\epsilon}{2}).$ Suppose that we want to solve
\begin{equation}
  x\pa_x^2 u=f(x)\text{\; with\; }u(0)=a\neq 0.
\end{equation}
As noted above, if $u\in\cC^2([0,\infty)),$ then, unless $f(0)=0,$ it cannot
simultaneously satisfy the differential equation as $x\to 0,$ and the boundary
condition. If we write
\begin{equation}
  u(x)=\psi(x)[f(0)x\log x+a]+v(x), 
\end{equation}
then
\begin{equation}
  x\pa_x^2u=x\psi''(x)[f(0)x\log x+a]+2x\psi'(x)[f(0)(\log x+1)]+\psi(x)f(0)+x\pa_x^2v.
\end{equation}
The first two terms on the right hand side are smooth and compactly supported
away from $x=0.$ The function $v$ must solve the equation
\begin{equation}
  x\pa_x^2v=f(x)-[x\psi''(x)[f(0)x\log x+a]+2x\psi'(x)[f(0)(\log x+1)]+\psi(x)f(0)],
\end{equation}
with $v(0)=0.$ The right hand side is smooth and vanishes at $x=0,$ hence the
theory presented in~\cite{WF1d} shows that there is a unique smooth
solution to this problem.

To carry this out in higher dimensions, we need to
introduce some notation. For indices $1\leq i_1<\cdots<i_k\leq n,$ (or $n+1$) we define
the vector valued function
\begin{equation}
  X_{\hi_1,\dots,\hi_k}=\sum_{\{j\notin\{i_1,\dots,i_k\}\}}x_je_j,
\end{equation}
where $e_j$ are the standard basis vectors for $\bbR^n$ (or $\bbR^{n+1}$ --- which
is needed will be clear from the context). For a function $\varphi$ defined in
$\bbR^n,$ the composition satisfies
\begin{equation}
  \varphi(X_{\hi_1,\dots,\hi_k})=\varphi\restrictedto_{x_{i_1}=\cdots=x_{i_k}=0}.
\end{equation}
Our method relies on the following essentially algebraic
lemma:
\begin{lemma}\label{lem2}
  Let $\eta(\tau)=\tau\log \tau,$ and let $L_{\bzero}$ denote the operator
  defined in~\eqref{modop0}. For any $k\leq n,$ and distinct indices
  $\{i_1,\dots,,i_k\}\subset\{1,\dots, n+1\},$ we have
  \begin{equation}
    L_{\bzero}\eta(x_{i_1}+\cdots+x_{i_k})=1.
  \end{equation}
\end{lemma}
\begin{proof}
By the permutation symmetry of the operator $L_{\bzero}$ it suffices to consider the indices
$\{1,\dots,k\}.$ Since $\pa_{\tau}^2\eta(\tau)=1/\tau,$ this follows from Lemma~\ref{lem1}.
\end{proof}

As the next step we write the function $v=v_0+v_1,$ where $v_0$ will be the
``regular'' part of the solution, vanishing on the boundary, and $v_1$ is the
singular part, which also vanishes on the boundary and satisfies the equation
\begin{equation}
  L_{\bzero}v_1\restrictedto_{bS_{n,0}}=\fone\restrictedto_{bS_{n,0}}.
\end{equation}
As we shall see, $v_1$ belongs to $\cC^{0,\gamma}_{\WF}(S_{n,0})$ for any
$0<\gamma<1.$ In fact we can simply write a formula for $v_1:$
\begin{equation}\label{singprt0}
  v_1(x)=\sum_{j=1}^n\sum_{\{1\leq i_1<\cdots<i_j\leq
    n\}}(-1)^{j-1}\fone(X_{\hi_1,\dots,\hi_j})
\eta(x_{i_1}+\cdots+x_{i_j}).
\end{equation}
For example, if $n=2,$ then
\begin{equation}
  v_1(x_1,x_2)=\fone(x_1,0)\eta(x_2)+\fone(0,x_2)\eta(x_1)-\fone(0,0)\eta(x_1+x_2);
\end{equation}
if $n=3,$ then
\begin{multline}
  v_1(x_1,x_2,x_3)=\fone(x_1,x_2,0)\eta(x_3)+\fone(x_1,0,x_3)\eta(x_2)+\fone(0,x_2,x_3)\eta(x_1)-\\
\fone(x_1,0,0)\eta(x_2+x_3)-\fone(0,x_2,0)\eta(x_1+x_3)-\\
\fone(0,0,x_3)\eta(x_1+x_2)+\fone(0,0,0)\eta(x_1+x_2+x_3).
\end{multline}
\begin{remark} Formula~\eqref{singprt0} should be contrasted to results like
  those in~\cite{costabel-dauge}, which analyzes boundary value problems for
  uniformly elliptic operators in domains with corners. For this case the
  classical Dirichlet and Neumann boundary value problems are well posed,
  however, even if the data is infinitely differentiable, these problems
  typically do \emph{not} have smooth solutions on domains whose boundaries
  have corner-type singularities. There is no analogue of the ``regular''
  solution, which is uniquely determined in the present case. The existence of
  a regular solution is another indication of the remarkable relationship
  between the degeneracies of the operator $\LKs{}$ and the singular structure
  of the boundary of the simplex. Singularities arise in the present case from
  the requirement that the solution satisfy a Dirichlet-type boundary
  condition, which, as we have seen, is generally impossible for a function in
  $\cC^2(\Sigma_n).$ 

  A second remarkable feature of this formula is that it provides two locally
  determined terms in the expansion, along $b\Sigma_n,$ of the solution $u$ to
  the original problem in~\eqref{dirprob0}. We have
\begin{equation}\label{eqn6.19}
  u(x)=\tg(x)+v_1(x)+O(\dist(x,b\Sigma_n)),
\end{equation}
where it should be noted that
$v_1(x)=O(\dist(x,b\Sigma_n)\log\dist(x,b\Sigma_n)).$ In the non-degenerate
case, determination of the second term in such an expansion requires the
solution of a global problem.
\end{remark}

\begin{theorem}\label{thm1} Let $\fone\in\cC^{0,2+\gamma}_{\WF}(\Sigma_n),$ then
  the function $v_1$ defined in~\eqref{singprt0} locally belongs to
  $\cC^{0,\gamma}_{\WF}(S_{n,0}),$ for any $0<\gamma<1,$ as does
  $L_{\bzero}v_1.$ It satisfies the following equations
  \begin{equation}
    L_{\bzero}v_1\restrictedto_{bS_{n,0}}=\fone\restrictedto_{bS_{n,0}}\text{ and\;\; }
v_1\restrictedto_{bS_{n,0}}=0.
  \end{equation}
\end{theorem}
\begin{proof}
The fact that $v_1\in \cC^{0,\gamma}_{\WF}(S_{n,0})$ follows immediately from
the fact that $\eta(\tau)$ is locally in $\cC^{0,\gamma}_{\WF}(S_{1,0}).$ We
use Lemma~\ref{lem2} and the fact that $X_{\hi_1,\dots,\hi_j}$ and
$x_{i_1}+\cdots+x_{i_j}$ depend on disjoint subsets of the variables to obtain that
\begin{equation}
  L_{\bzero}v_1=\sum_{j=1}^n\sum_{\{1\leq i_1<\cdots<i_j\leq
    n\}}(-1)^{j-1}\left[L_{\bzero}\fone(X_{\hi_1,\dots,\hi_j})
\eta(x_{i_1}+\cdots+x_{i_j})+\fone(X_{\hi_1,\dots,\hi_j})\right]
\end{equation}
It is clear that $L_{\bzero}v_1\in\cC^{0,\gamma}_{\WF}(S_{n,0})$ for any
$0<\gamma<1.$ The facts that $ v_1\restrictedto_{bS_{n,0}}=0,$ and $
L_{\bzero}v_1\restrictedto_{bS_{n,0}}=\fone\restrictedto_{bS_{n,0}}$ follow
from a direct calculation. We first give the proof for the $n=2$ case:
Observe that
\begin{equation}
  v_1(x_1,x_2)=\fone(x_1,0)\eta(x_2)+\fone(0,x_2)\eta(x_1)-\fone(0,0)\eta(x_1+x_2).
\end{equation}
We restrict to $x_2=0$ to obtain
\begin{equation}
  v_1(x_1,0)=\fone(0,0)\eta(x_1)-\fone(0,0)\eta(x_1)=0.
\end{equation}
The case $x_1=0$ is identical. Applying $L_{\bzero}$ we see that
\begin{multline}
  L_{\bzero}v_1(x_1,x_2)=x_1\pa_{x_1}^2\fone(x_1,0)\eta(x_2)+\\
\fone(x_1,0)+x_2\pa_{x_2}^2\fone(0,x_2)\eta(x_1)+
\fone(0,x_2)-\fone(0,0).
\end{multline}
Setting $x_2=0$ now gives
\begin{equation}
  L_{\bzero}v_1(x_1,0)=\fone(x_1,0)+\fone(0,0)-\fone(0,0)=\fone(x_1,0).
\end{equation}

The general case is not much harder: By symmetry, and continuity it suffices to
show that
\begin{equation}\label{eqn3.23.01}
\begin{split}
  &\sum_{j=1}^n\sum_{\{1\leq i_1<\cdots<i_j\leq
    n\}}(-1)^{j-1}\fone(X_{\hi_1,\dots,\hi_j})
\eta(x_{i_1}+\cdots+x_{i_j})\restrictedto_{x_n=0}=0\quad\text{ and }\\
&\sum_{j=1}^n\sum_{\{1\leq i_1<\cdots<i_j\leq
    n\}}(-1)^{j-1}\left[L_{\bzero}\fone(X_{\hi_1,\dots,\hi_j})
\eta(x_{i_1}+\cdots+x_{i_j})+\fone(X_{\hi_1,\dots,\hi_j})\right]_{x_n=0}\\&=f(X_{\hn}).
\end{split}
\end{equation}
We prove the first identity in~\eqref{eqn3.23.01} by observing that, if
$i_j<n,$ then
$X_{\hi_1,\dots,\hi_j}\restrictedto_{x_n=0}=X_{\hi_1,\dots,\hi_j,\hn}.$ We
split the sum into two parts and use the fact that
$\eta(x_n)\restrictedto_{x_n=0}=0$ to obtain:
\begin{multline}\label{eqn3.24.01}
  \sum_{j=1}^n\sum_{\{1\leq i_1<\cdots<i_j\leq
    n\}}(-1)^{j-1}\fone(X_{\hi_1,\dots,\hi_j})
\eta(x_{i_1}+\cdots+x_{i_j})\restrictedto_{x_n=0}=\\
\sum_{j=1}^{n-1}\sum_{\{1\leq i_1<\cdots<i_j\leq
    n-1\}}(-1)^{j-1}\fone(X_{\hi_1,\dots,\hi_j,\hn})
\eta(x_{i_1}+\cdots+x_{i_j})+\\
\sum_{j=2}^n\sum_{\{1\leq i_1<\cdots<i_{j-1}\leq
    n-1\}}(-1)^{j-1}\fone(X_{\hi_1,\dots,\hi_{j-1},\hn})
\eta(x_{i_1}+\cdots+x_{i_{j-1}}).
\end{multline}
Upon changing $j-1\to j$ in the third line, it becomes clear that the second
and third lines in~\eqref{eqn3.24.01} differ only by a sign, and therefore sum
to zero.

To prove the second identity in~\eqref{eqn3.23.01}, we observe that the first
identity already implies that
\begin{equation}
  \sum_{j=1}^n\sum_{\{1\leq i_1<\cdots<i_j\leq
    n\}}(-1)^{j-1}L_{\bzero}\fone(X_{\hi_1,\dots,\hi_j})
\eta(x_{i_1}+\cdots+x_{i_j})\restrictedto_{x_n=0}=0,
\end{equation}
so we are only left to prove that
\begin{equation}
  \sum_{j=1}^n\sum_{\{1\leq i_1<\cdots<i_j\leq
    n\}}(-1)^{j-1}\fone(X_{\hi_1,\dots,\hi_j})\restrictedto_{x_n=0}=f(X_{\hn}).
\end{equation}
Applying the same decomposition as before, we see that
\begin{multline}
   \sum_{j=1}^n\sum_{\{1\leq i_1<\cdots<i_j\leq
    n\}}(-1)^{j-1}\fone(X_{\hi_1,\dots,\hi_j})\restrictedto_{x_n=0}=\\
 \fone(X_{\hn})+\sum_{j=1}^{n-1}\sum_{\{1\leq i_1<\cdots<i_j\leq
    n-1\}}(-1)^{j-1}\fone(X_{\hi_1,\dots,\hi_j,\hn})+\\
 \sum_{j=2}^n\sum_{\{1\leq i_1<\cdots<i_{j-1}\leq
    n-1\}}(-1)^{j-1}\fone(X_{\hi_1,\dots,\hi_{j-1},\hn}),
\end{multline}
from which the claim is again immediate.
\end{proof}

This almost suffices, except for the fact that $v_1$ is not generally compactly
supported even if the data is. This is because in the definition of $v_1$ we
evaluate $\fone$ at subsets of the variables $(x_1,\dots,x_n).$ To repair this
we multiply $v_1$ by a specially constructed bump function. To see that this
works we need  another elementary calculus lemma.
\begin{lemma}\label{lem3}
  Let $\psi\in\cC^2([0,\infty))$ and $h\in\cC^2(\Int(S_{n,0})),$ then
  \begin{multline}
    L_{\bzero}[h\psi(x_1+\cdots+x_n)]=\\\psi(x_1+\cdots+x_n) L_{\bzero}h+
2\psi'(x_1+\cdots+x_n) Rh+(x_1+\cdots+x_n)\psi''(x_1+\cdots+x_n)h.
  \end{multline}
Here $R=x_1\pa_{x_1}+\cdots+x_n\pa_{x_n}.$
\end{lemma}

Using this lemma we can easily demonstrate the following result:
\begin{proposition}
  Suppose that $\fone\in\cC^{0,2+\gamma}_{\WF}(S_{n,0})$ is supported in the set $\{x:\: x_1+\cdots+x_n\leq N\}.$
  If $\psi\in\cC^2_c([0,\infty))$ equals $1$ in $[0,N],$ then
  \begin{equation}
    \tv_1(x)=\psi(x_1+\cdots+x_n)v_1(x)
  \end{equation}
is compactly supported, $\tv_1$ and $L_{\bzero}\tv_1$ belong to
$\cC^{0,\gamma}_{\WF}(S_{n,0}).$ The function $\tv_1$ also satisfies the equations:
 \begin{equation}
    L_{\bzero}\tv_1\restrictedto_{bS_{n,0}}=\fone\restrictedto_{bS_{n,0}}\text{ and\;\; }
\tv_1\restrictedto_{bS_{n,0}}=0.
  \end{equation}
\end{proposition}
\begin{proof}
The fact that $\tv_1\restrictedto_{bS_{n,0}}=0$ and its regularity properties are immediate. Applying
Lemma~\ref{lem3} we see that
 \begin{equation}
    L_{\bzero}\tv_1=\psi(x_1+\cdots+x_n)L_{\bzero}v_1+\
2\psi'(x_1+\cdots+x_n) Rv_1+(x_1+\cdots+x_n)\psi''(x_1+\cdots+x_n)v_1.
\end{equation}
The fact that $\tau\pa_{\tau}\eta(\tau)=\tau(\log\tau+1)$ implies that $Rv_1$
has the same regularity as $v_1,$ and therefore so does $L_{\bzero}\tv_1.$ Thus to prove
the proposition we only need to show that $Rv_1$ vanishes on the interiors of
the hypersurface boundary components. In a neighborhood of the interior of
$x_1=0$ we have the representation $v_1=x_1\log x_1a_1(x)+x_1a_2(x),$ where $a_1$ and $a_2$ belong
locally to $\cC^2.$ Applying the vector field we see that
\begin{equation}
  Rv_1= x_1(\log x_1+1)a_1(x)+x_1\log x_1Ra_1(x)+x_1(a_2(x)+Ra_2),
\end{equation}
which obviously vanishes as $x_1\to 0.$ The other boundary faces follow by an
essentially identical argument.
\end{proof}

We now can complete the solution of~\eqref{dirprob00} and~\eqref{dirprob0} as
well. Write $v=\tv_0+\tv_1,$ where
\begin{equation}\label{eqn6.36}
  L_{\bzero}\tv_0=\fone- L_{\bzero}\tv_1\text{\; and\; }\tv_0\restrictedto_{bS_{n.0}}=0.
\end{equation}
Since $\fone- L_{\bzero}\tv_1\in\cC^{0,\gamma}_{\WF}(S_{n,0})$ is compactly
supported and vanishes on the boundary, it follows from the results
in~\cite{EpMaz2} that~\eqref{eqn6.36} has a unique solution
$\tv_0\in\cC^{0,2+\gamma}_{\WF}(S_{n,0}).$ While we call $\tv_0$ the regular
part of the solution, it will not in general be smooth. It will have
complicated singularities of the form
$[x_{i_1}+\cdots+x_{i_k}]^l[\log(x_{i_1}+\cdots+x_{i_k})]^m,$ for $2\leq l,$
and with the maximum value of $m$ bounded by a function of $l.$

With  small adaptations this method can also be used to treat the case of the
$n$-simplex, $\Sigma_n,$ and the neutral Kimura diffusion operator, $\LK.$
As noted above, the $n$-simplex is most symmetrically viewed in the affine plane
$x_1+\cdots+x_{n+1}=1.$ This representation makes clear that every vertex is
identical to every other vertex. For the construction of the solution to the
Dirichlet problem:
\begin{equation}\label{eqn6.37}
  \LK u=f\text{ in }\Sigma_n\text{\; with\; }u\restrictedto_{b\Sigma_n}=g
\end{equation}
it turns out to be simplest to work in the non-symmetric representation, with
one vertex identified with $\bzero.$ Recall that
\begin{equation}
  \tSigma_n=\{(x_1,\dots,x_n):\:0\leq x_i,\, i=1,\dots,n\,;\, x_1+\cdots+x_n\leq 1\}.
\end{equation}
It is clear that there is a choice of linear projection into $\bbR^n$ so that
any given vertex of $\Sigma_n$ is so identified with $\bzero.$ For the moment
we work in $\tSigma_n$ with coordinates $(x_1,\dots,x_n).$

We begin by assuming that the data $f,g$ is supported in a set of the form
$\tSigma_{n,2\epsilon}=\{x:\: 0\leq x_1+\cdots+x_n\leq 1-2\epsilon\},$ for some $0<\epsilon.$ As
before we let $\tg$ denote an  optimally smooth extension of $g$ as a
function with support in $\tSigma_{n,\epsilon}.$ We write $u=\tg+v,$ where $v$ satisfies
\begin{equation}\label{eqn3.32}
  \LK v=f-\LK\tg=\fone,\text{ with\; }v\restrictedto_{b\tSigma_n}=0.
\end{equation}
As before we write $v=v_0+\tv_1,$ where
\begin{equation}
  \tv_1=\psi(x_1+\dots+x_n)v_1(x),
\end{equation}
with $\psi\in\cC^{\infty}_{c}([0,1-\frac{\epsilon}{2})),$ satisfying $\psi(\tau)=1,$
for $\tau\in [0,1-\epsilon].$ The use of the bump function is much
more critical here, as we do not have good control on the function $v_1$ on
the face where $x_1+\cdots+x_n=1.$  

 The function $v_1$ is defined by the sum:
\begin{equation}\label{singprtsplx}
  v_1(x)=\sum_{j=1}^n\sum_{\{1\leq i_1<\cdots<i_j\leq
    n\}}(-1)^{j-1}\fone(X_{\hi_1,\dots,\hi_j})
\eta(x_{i_1}+\cdots+x_{i_j}).
\end{equation}
What is special about the choice of coordinates is the possibility of having
the two functions, $\fone(X_{\hi_1,\dots,\hi_j})$ and $\eta(x_{i_1}+\cdots+x_{i_j}),$
depend on disjoint sets of coordinates. Elementary calculations show that
\begin{equation}
  \LK \eta(x_{i_1}+\cdots+x_{i_j})=1-(x_{i_1}+\cdots+x_{i_j}),
\end{equation}
and
\begin{multline}\label{eqn3.35}
  \LK [\fone(X_{\hi_1,\dots,\hi_j})\eta(x_{i_1}+\cdots+x_{i_j})]=
\eta(x_{i_1}+\cdots+x_{i_j})\LK \fone(X_{\hi_1,\dots,\hi_j})-\\
2(x_{i_1}+\cdots+x_{i_j})\eta'(x_{i_1}+\cdots+x_{i_j})
R_{\hi_1,\dots,\hi_j}\fone(X_{\hi_1,\dots,\hi_j})
+\\
+(1-(x_{i_1}+\cdots+x_{i_j}))\fone(X_{\hi_1,\dots,\hi_j}),
\end{multline}
where the vector fields $R_{\hi_1,\dots,\hi_j}$ are defined by
\begin{equation}
  R_{\hi_1,\dots,\hi_j}=\sum_{k\notin\{i_i,\dots,i_j\}}x_k\pa_{x_k}.
\end{equation}

Let $b\tSigma_n'=b\tSigma_n\setminus\{x:\: x_1+\cdots+x_n=1\}.$ Arguing as
before we show that $v_1$ satisfies the equation along $b\tSigma_n'.$
\begin{theorem} Let $\fone\in\cC^{0,2+\gamma}_{\WF}(\Sigma_n).$ The function
  $v_1$ defined in~\eqref{singprtsplx} belongs to
  $\cC^{0,\gamma}_{\WF}(\tSigma_n),$ as does $\LK v_1.$ It satisfies the
  following equations
  \begin{equation}\label{smplxeqns}
    \LK v_1\restrictedto_{b\tSigma_n'}=\fone\restrictedto_{b\tSigma_n'}\text{ and\;\; }
v_1\restrictedto_{b\tSigma_n'}=0.
  \end{equation}
\end{theorem}
\begin{proof}
  The regularity statements for $v_1$ and $\LK v_1$ follow easily
  from~\eqref{singprtsplx} and~\eqref{eqn3.35}. The fact that
  $v_1\restrictedto_{b\tSigma_n'}=0$ follows from Theorem~\ref{thm1}. The proof
  that $\LK v_1\restrictedto_{b\tSigma_n'}=\fone\restrictedto_{b\tSigma_n'}$ is
  similar to the proof of Theorem~\ref{thm1}. As before continuity shows that
  we only need to prove this statement for the interiors of the hypersurface
  faces, and symmetry shows that it suffices to consider $\{x_n=0\}.$ The terms
  coming from the first line of~\eqref{eqn3.35} can be treated exactly as
  before. For the terms from the third line of~\eqref{eqn3.35}, the sole
  difference is the coefficient $(1-x_n),$ which equals $1$ where $x_n=0.$ The
  first order cross terms, coming from the second line of~\eqref{eqn3.35},
  require some additional consideration.

We once again use the observation that if
$i_j<n,$ then
$X_{\hi_1,\dots,\hi_j}\restrictedto_{x_n=0}=X_{\hi_1,\dots,\hi_j,\hn},$ as well
as the facts that
$$R_{\hi_1,\dots,\hi_j}\fone(X_{\hi_1,\dots,\hi_j})\restrictedto_{x_n=0}=
R_{\hi_1,\dots,\hi_j,\hn}\fone(X_{\hi_1,\dots,\hi_j,\hn}),$$ 
and 
$$x_n\pa_{x_n}\eta(x_n)\restrictedto_{x_n=0}=(x_n\log
x_n+x_n)\restrictedto_{x_n=0}=0,$$ 
to split the sum into two parts, obtaining:
\begin{multline}\label{eqn3.24.05}
  \sum_{j=1}^n\sum_{\{1\leq i_1<\cdots<i_j\leq
    n\}}(-1)^{j-1}2(x_{i_1}+\cdots+x_{i_j})\eta'(x_{i_1}+\cdots+x_{i_j})
R_{\hi_1,\dots,\hi_j}\fone(X_{\hi_1,\dots,\hi_j})\restrictedto_{x_n=0}
=\\
\sum_{j=1}^{n-1}\sum_{\{1\leq i_1<\cdots<i_j\leq
    n-1\}}(-1)^{j-1}2(x_{i_1}+\cdots+x_{i_j})\eta'(x_{i_1}+\cdots+x_{i_j})
R_{\hi_1,\dots,\hi_j,\hn}\fone(X_{\hi_1,\dots,\hi_j,\hn})+\\
\sum_{j=2}^n\sum_{\{1\leq i_1<\cdots<i_{j-1}\leq
    n-1\}}(-1)^{j-1}2(x_{i_1}+\cdots+x_{i_j})\eta'(x_{i_1}+\cdots+x_{i_j})
R_{\hi_1,\dots,\hi_j,\hn}\fone(X_{\hi_1,\dots,\hi_j,\hn}).
\end{multline}
As before, after changing variables in the third line with $j\mapsto j-1,$
we see that the  second and third lines differ only by a sign, demonstrating
that these terms sum to zero along the face given by $\{x_n=0\}.$ This
completes the proof of the theorem.
\end{proof}

To complete the discussion of this case we need to check that $\tv_1$ also
satisfies the equations in~\eqref{smplxeqns}. An elementary calculation shows
that
\begin{multline}
  \LK[v_1\psi(x_1+\cdots+x_n)]=\\
\psi(x_1+\cdots+x_n)\LK
  v_1+\psi'(x_1+\cdots+x_n)\tR v_1+
v_1\LK\psi,
\end{multline}
where
\begin{equation}
  \tR = 2(1-(x_1+\cdots+x_n))\sum_{j=1}^nx_i\pa_{x_i}.
\end{equation}
As before, $\tR$ is tangent to $b\tSigma_n,$ and
$\tR v_1\in\cC^{0,\gamma}_{\WF}(\tSigma_n).$ It is easy to see that 
\begin{equation}
  \tR v_1\restrictedto_{b\tSigma_n'}=0.
\end{equation}
Summarizing, we have shown
\begin{proposition} Let $\fone\in\cC^{0,2+\gamma}_{\WF}(\Sigma_n).$ 
  The function $\tv_1$ belongs to $\cC^{0,\gamma}_{\WF}(\tSigma_n),$ as does
  $\LK \tv_1.$ If $\psi(x_1+\cdots+x_n)=1$ on
  $\supp\fone\subset\tSigma_{n,\epsilon},$ for an $\epsilon>0,$ then $\tv_1$
  satisfies the following equations
  \begin{equation}\label{smplxeqns2}
    \LK \tv_1\restrictedto_{b\tSigma_n}=\fone\restrictedto_{b\tSigma_n}\text{ and\;\; }
\tv_1\restrictedto_{b\tSigma_n}=0.
\end{equation}
\end{proposition}

We can now complete the solution of~\eqref{eqn3.32}. We write $v=v_0+\tv_1,$
the function $v_0$ must satisfy
\begin{equation}\label{eq:v0:correction}
  \LK v_0=\fone-\LK\tv_1\text{\; with\; }v_0\restrictedto_{b\Sigma_n}=0.
\end{equation}
The existence of a solution $v_0\in\cC^{0,2+\gamma}_{\WF}(\tSigma_n),$ which can
be taken to vanish on the boundary, follows from the results in~\cite{EpMaz2}. 
Numerically, the method of Section~\ref{sec:numerics} (with $k=n$) can
  be used to solve (\ref{eq:v0:correction}) for $v_0$.

Finally if we have general data $(f,g)$ so that $f,$ and $\LK\tg\in\cC^{0,2+\gamma}_{\WF}(\Sigma_n),$ then
we choose a partition of unity $\{\varphi_1,\dots,\varphi_{n+1}\}$ so
that the function $\varphi_j$ equals $1$ in a neighborhood of the
vertex $e_j\subset\Sigma_n\subset\bbR^{n+1},$ and vanishes in a
neighborhood of the opposite face, where $x_j=0.$ The data
$\{\varphi_j(f,\tg):\: j=1,\dots,n+1\}$ satisfies the hypotheses used
above, and therefore we can construct solutions $\{u_j\}$ to the
equations
\begin{equation}
  \LK u_j=\varphi_j f\text{ on }\Sigma_n,\text{ with\; }u_j\restrictedto_{b\Sigma_n}=\varphi_j g,
\text{\; for\, }j=1,\dots,n+1.
\end{equation}
To construct the solution in a neighborhood of $e_j\in\bbR^{n+1}$ we project the simplex
into the hyperplane $\{x_j=0\},$ and use the projective representation where
$e_j$ corresponds to $\bzero.$ 

Now setting
\begin{equation}
  u=u_1+\cdots+u_{n+1},
\end{equation}
we obtain a solution to the original boundary value problem
\begin{equation}\label{eqn4.51}
  \LK u=f\text{ on }\Sigma_n\text{\; with\; }u\restrictedto_{b\Sigma_n}=g.
\end{equation}
This proves the following general result:
\begin{theorem}
  Let $f\in\cC^{0,2+\gamma}_{\WF}(\Sigma_n),$ and $g$ have an extension $\tg$
  to $\Sigma_n$ so that
  $\LK\tg\in\cC^{0,2+\gamma}_{\WF}(\Sigma_n).$ The Dirichlet
  problem~\eqref{eqn4.51} has unique solution $u,$ which takes the form
  $u=\tg+u^{(0)}+u^{(1)},$ where $u^{(0)}\in\cC^{0,2+\gamma}_{\WF}(\Sigma_n),$
  vanishes on the boundary. The singular part, $u^{(1)}$ is of the
  form
\begin{equation}\label{eqn3.48}
  u^{(1)}(x)=\sum_{i=1}^{n+1}\sum_{j=1}^n
\sum_{\{1\leq i_1<\cdots<i_j\leq
    n+1:\:i_m\neq i\}}(-1)^{j-1}F(X_{\{\hi_1,\dots,\hi_j\}})\eta(x_{i_1}+\cdots+x_{i_j})
\end{equation}
Here $F$ is a function constructed from the pullbacks of $\{\varphi_j\fone\}$
to the affine model $\Sigma_n.$
\end{theorem}
\begin{remark}
  The solution $u$ to the boundary value problem in~\eqref{eqn4.51} also has an
  explicit 2-term expansion at the boundary similar to that given
  in~\eqref{eqn6.19}:
  \begin{equation}\label{eqn6.19.2}
    u(x)=\tg(x)+u^{(1)}(x)+O(\dist(x,b\Sigma_n)).
  \end{equation}
\end{remark}
\begin{proof}
Everything has been proved except the uniqueness statement.
 This follows from the maximum principle and the facts that $u$ is
continuous in the closed simplex, and $\LK$ is a strongly elliptic operator in
the interior of the simplex.
\end{proof}

We observe that $u$ has, in some sense, very complicated singularities, in that
it includes a smooth function times
$(x_{i_1}+\cdots+x_{i_j})\log(x_{i_i}+\cdots+x_{i_j}),$ for each set of indices
\begin{equation*}
1\leq i_1<\cdots<i_j\leq n+1, \text{\; for\; } j\in\{1,\dots,n\}.
\end{equation*}
To resolve these
singularities to be classically conormal, one would, in principle, need to
successively blow up all the strata of the boundary, starting with the
codimension $n$ parts and proceeding upwards to the codimension $2$ part.
Given the very explicit form that this singularity takes, such an approach
would only obscure its simple and rather benign structure. This general
approach is pursued in the papers~\cite{Jost1, Jost2, Jost3}. These authors
do not require the data $(f,g)$ to be continuous, but allow data with
complicated singularities along the boundary.

We remark that the approach proposed in Section~\ref{sec3} for
  solving the regular problem $\LK u=f$ (without boundary conditions)
  can be modified slightly to provide an algorithm for extending $g$ (defined on
  the boundary) to $\tg$ (in the interior) in such a way that $\LK\tg$
  is very easily computed.  We write $g=g_0+\cdots+g_{n-1}$, where
  $g_0=\sum_{j=1}^{n+1} g(e_j)x_j$ agrees with $g$ on $\Sigma_n^0$,
  and
\begin{equation*}
  g_k\restrictedto_{\Sigma_n^k} = g\restrictedto_{\Sigma_n^k} -
  \sum_{j=0}^{k-1}g_j\restrictedto_{\Sigma_n^k}, \qquad
  (1\le k\le n-1).
\end{equation*}
The right-hand side is zero on $\Sigma_n^{k-1}$, so this equation decouples
into independent homogeneous ``Dirichlet'' extension problems from the
connected components of $\Sigma_n^k\setminus\Sigma_n^{k-1}$ to $\Sigma_n$.

Using the eigenfunction expansion to represent the right-hand side on each face
of $\Sigma_n^k$ gives the desired representations
\begin{equation}
 g_k=\sum_{\cI} c_{\cI,\vec
  m} (w_{\cI}\psi_{\cI,\vec m})\text{ and }\LK g_k=\sum_{\cI} (\lambda_{\cI,\vec
  m}c_{\cI,\vec m})(w_{\cI}\psi_{\cI,\vec m}),
\end{equation}
 with both functions canonically
defined throughout $\Sigma_n.$ If $g$ is a polynomial of
degree less than or equal to $d$ on each face of $b\Sigma_n,$ then
$\tg\in\cP_d$ as well. Thus, we can replace $f$ by $f^{(1)}=f-\LK\tg$ in
(\ref{eqn3.32}) at the outset, thereby avoiding
non-homogeneous boundary conditions when working with the partition of unity.

Our approach to solving the Dirichlet problem works equally well if we add a
vector field $V$ that is everywhere tangent to $b\Sigma_n.$ In a projective
chart, such a vector field takes the form
\begin{equation}
  \tV=\sum_{j=1}^{n}b_j(x)x_j\pa_{x_j},
\end{equation}
with the additional requirement that
\begin{equation}
  \sum_{j=1}^{n}x_jb_j(x)\restrictedto_{x_1+\cdots+x_n=1}=0.
\end{equation}

A simple calculation shows that
\begin{equation}\label{eqn3.51}
  \tV\eta(x_{i_1}+\cdots+x_{i_j})=\sum_{l=1}^na_{l}(x)x_l\Delta_{\{i_1,\dots,i_j\}}(l)[\log(x_{i_1}+\cdots+x_{i_j})+1],
\end{equation}
where
\begin{equation}
  \Delta_{\{i_1,\dots,i_j\}}(l)=\begin{cases}1\text{ if }&l\in
    \{i_1,\dots,i_j\}\\
0\text{ if }&l\notin
    \{i_1,\dots,i_j\}.
\end{cases}
\end{equation}
The function on the right hand side of~\eqref{eqn3.51} is easily seen to be
continuous on the closed $n$-simplex. In fact these functions belong to
$\cC^{0,\gamma}_{\WF}(\Sigma_n),$ for any $0<\gamma<1.$ Applying $\tV$ to $\tv_1$ we see that
\begin{equation}\label{eqn3.53}
  \tV\tv_1=v_1 \tV\psi(x_1+\cdots+x_n)+\psi(x_1+\cdots+x_n)\tV v_1,
\end{equation}
from which it is clear that $\tV\tv_1$ is continuous on $\Sigma_n.$  To show
that $ \tV\tv_1\restrictedto_{b\Sigma_n}=0,$ it therefore suffices to prove it
in the interiors of the hypersurface boundary faces, but this is clear
from~\eqref{eqn3.51} and~\eqref{eqn3.53}. With these observations it follows
that $(\LK+\tV)\tv_1\in\cC^{0,\gamma}_{\WF}(\Sigma_n)$ and
\begin{equation}
  (\LK+\tV)\tv_1\restrictedto_{b\Sigma_n}=\LK\tv_1\restrictedto_{b\Sigma_n}=\fone\restrictedto_{b\Sigma_n}.
\end{equation}

Proceeding as above we easily demonstrate that, if $V$ is tangent to the boundary
$b\Sigma_n,$ then the Dirichlet problem:
\begin{equation}
  (\LK+V)u=f\text{ in }\Sigma_n\text{\; with\; }u\restrictedto_{b\Sigma_n}=g
\end{equation}
has a unique solution of the form $u=u_0+u_1,$ where $u_1$ is given by the
formula in~\eqref{eqn3.48}, and $u_0\in\dcC^{0,2+\gamma}_{\WF}(\Sigma_n).$


\end{document}